\documentclass[final,hidelinks,onefignum,onetabnum]{siamart251216}

\usepackage{lipsum}
\usepackage{amsfonts}
\usepackage{graphicx}
\usepackage{epstopdf}
\usepackage{algorithmic}
\usepackage{bm}
\usepackage{subcaption}
\ifpdf
  \DeclareGraphicsExtensions{.eps,.pdf,.png,.jpg}
\else
  \DeclareGraphicsExtensions{.eps}
\fi

\newsiamremark{remark}{Remark}
\newsiamremark{hypothesis}{Hypothesis}
\crefname{hypothesis}{Hypothesis}{Hypotheses}
\newsiamthm{claim}{Claim}
\newsiamremark{fact}{Fact}
\crefname{fact}{Fact}{Facts}

\headers{Loss Landscape of Deep Matrix Factorization}{P. Chen, R. Jiang, and P. Wang}

\title{A Complete Loss Landscape Analysis of Regularized \\ Deep Matrix Factorization\thanks{The authors are listed alphabetically.
\funding{P. Chen and R. Jiang are supported in part by the National Key R\&D Program of China under grant 2023YFA1009300 and in part by the Major Program of NSFC 72394360 and 72394364. P. Wang is supported in part by the University of Macau SRG2025-00043-FST and UMDF-TISF-I/2026/013/FST, and in part by the Macau Science and Technology Development Fund (FDCT) 0091/2025/ITP2.}}}

\author{Po Chen\thanks{School of Data Science, Fudan University, Shanghai, China (\email{chenp24@m.fudan.edu.cn}).}
\and Rujun Jiang\thanks{School of Data Science, Fudan University, Shanghai, China (\email{rjjiang@fudan.edu.cn}).}
\and Peng Wang\thanks{Corresponding author. Department of Computer and Information Science, University of Macau, Macau SAR, China (\email{pengw@um.edu.mo}).}}

\usepackage{amsopn}

\def\bmw{\bm{W}}

\def\bU{\bm{U}}
\def\bV{\bm{V}}

\def\mO{\mathcal{O}}

\def\R{\mathbb{R}}

\newcommand{\blk}{\mathrm{BlkD}}

\newcommand{\retwo}[1]{{ #1}}

\ifpdf
\hypersetup{
  pdftitle={Loss Landscape of Deep Matrix Factorization},
  pdfauthor={P. Chen, R. Jiang, and P. Wang}
}
\fi

\begin{document}

\maketitle
\begin{abstract}
Despite its wide range of applications across various domains, the optimization foundations of deep matrix factorization (DMF) remain largely open. In this work, we aim to fill this gap by conducting a comprehensive study of the loss landscape of the regularized DMF problem. Toward this goal, we first provide a closed-form characterization of all critical points of the problem. Building on this, we establish precise conditions under which a critical point is a local minimizer, a global minimizer, a strict saddle point, or a non-strict saddle point. Leveraging these results, we derive a necessary and sufficient condition under which every critical point is either a local minimizer or a strict saddle point. This provides insights into why gradient-based methods almost always converge to a local minimizer of the regularized DMF problem. Finally, we conduct numerical experiments to visualize its loss landscape to support our theory.
\end{abstract}
\vspace{-0.1in}

\vspace{-0.1in}

\section{Introduction}\label{sec:intro}

Deep matrix factorization (DMF) is a powerful framework for learning hierarchical representations of data and has found wide applications across various domains, such as recommendation systems \cite{he2017neural,xue2017deep}, computer vision \cite{trigeorgis2016deep}, and bioinformatics \cite{baptista2021deep}, to name a few. Typically, DMF extends classical matrix factorization (MF) by decomposing a target matrix into a product of more than two matrices. While the optimization foundations of classical MF have been extensively studied in the literature (see, e.g., \cite{li2019symmetry,zhu2018global}), the theoretical understanding of DMF remains relatively limited. To lay down the optimization foundations of DMF, we study the following regularized formulation:
\begin{align}\label{eq:F}
\min_{\bm W} F(\bm W) :=  \left\| \bm W_L\cdots\bm W_1 -  \bm Y \right\|_F^2 + \sum_{l=1}^L \lambda_l \|\bm W_l\|_F^2,
\end{align}
where $\bm Y \in \R^{d_L \times d_0}$ denotes the data input, $L \ge 2$ is the number of factors (i.e., layers), $\bm W_l \in \R^{d_l \times d_{l-1}}$ denotes the $l$-th weight matrix for each $l=1,\dots,L$, $\bm W = \{\bm W_l\}_{l=1}^L$ denotes the collection of all weight matrices, and $\lambda_l > 0$ for each $l=1,\dots,L$ are regularization parameters.

In modern non-convex optimization, analyzing the loss landscape plays a key role in deepening our understanding of optimization algorithms. In particular, many machine learning models, such as deep neural networks, often have abundant non-strict saddle points (see \Cref{def:crit}) \cite{achour2024loss}. These saddle points can trap first-order methods in suboptimal regions, limiting their effectiveness. Fortunately, a growing body of work has shown that, despite their non-convexity, many non-convex optimization problems arising in machine learning possess a {\em benign loss landscape} that satisfies the following properties: (i) there are no {\em spurious local minimizers}, i.e., all local minimizers are global minimizers and (ii) each saddle point of the objective function is a strict saddle point (see \Cref{def:crit}). Notable examples of such problems include phase retrieval \cite{sun2018geometric}, low-rank matrix recovery \cite{bhojanapalli2016global,ge2017no}, dictionary learning \cite{sun2015complete,sun2016complete}, neural collapse \cite{yaras2022neural,zhu2021geometric}, among others. Additionally, some non-convex problems exhibit a {\em partially benign} loss landscape, where only one of these two desirable properties holds. For example, it has been shown that group synchronization  \cite{ling2023solving,ling2025local,ling2019landscape,mcrae2024benign} and Tucker decomposition  \cite{frandsen2022optimization} problems have no spurious local minimizers. In another case, tensor decomposition \cite{ge2015escaping} and maximal coding reduction problems \cite{wang2024global} satisfy the strict saddle property, whose critical points only consist of local minimizers and strict saddle points. In this work, we contribute to this line of research by studying the loss landscape of the DMF problem \eqref{eq:F}, with the goal of determining whether it also exhibits a (partially) benign landscape.

In general, the DMF problem is highly complex and difficult to analyze due to its non-convex and hierarchical nature. Nevertheless, substantial progress has been made in recent years toward studying the loss landscape across different settings. A line of research has studied the classic MF problem, i.e., $L=2$. For example, the work \cite{li2019symmetry} studied a low-rank MF problem with symmetric structures and showed that each critical point is either a global minimizer or a strict saddle point. The work \cite{ge2017no} developed a unified framework to study the loss landscape of low-rank matrix problems, including matrix sensing, matrix completion, and robust PCA, and arrived at the same conclusion. In addition, the work \cite{zhou2022optimization} conducted a complete loss landscape analysis of Problem \eqref{eq:F} for the case $L=2$, demonstrating that any critical point that is not a global minimizer is a strict saddle point. {  Moreover, the work \cite{kunin2019loss} analyzed the loss landscape of regularized linear autoencoders  with depth $L=2$ and provided a complete characterization of the loss landscape.} However, all the above analyses are limited to two-layer factorization. Extending such analyses to deeper architectures (i.e., $L\ge 3$) remains a challenging yet important direction for fully understanding the optimization foundations of DMF.

Another independent line of research has focused on analyzing the loss landscape of the unregularized version of Problem \eqref{eq:F}.
For example, several works \cite{baldi1989neural,freeman2016topology,kawaguchi2016deep,lu2017depth,nouiehed2022learning,Trager2020Pure} studied the unregularized loss surface of deep linear networks and proved that each critical point is either a global minimizer or a saddle point under mild conditions.
Building on this, recent studies \cite{laurent2018deep,Trager2020Pure} extended the analysis to deep linear networks with arbitrary convex differentiable loss and showed that all local minima are global minima if hidden layers are as wide as the input and output layers.
Notably, the work \cite{achour2024loss} provided a complete analysis of the loss landscape of the unregularized version of Problem \eqref{eq:F}. Specifically, they identified conditions under which each critical point is a global minimizer, a strict saddle point, or a non-strict saddle point by conducting a second-order analysis.
Despite the inspiring progress, it remains underexplored in the literature to study the loss landscape of the regularized problem \eqref{eq:F}.

Beyond the aforementioned lines of research, other works have also contributed to the analysis of the loss landscape of deep linear networks. The work \cite{mehta2021loss} demonstrated that introducing regularization can lead to spurious local minima using computational algebraic geometry. The work \cite{ziyin2022exact} derived analytical expressions for the global minima of a regularized deep linear network and showed that the zero solution constitutes a spurious local minimum. A recent work \cite{wang2023implicit} studied deep linear networks for matrix completion and showed that this problem admits multiple spurious local minima of different ranks.
In parallel, many works have investigated the training dynamics of deep linear networks, focusing on implicit bias \cite{arora2019implicit,chou2024gradient} and the effects of initialization \cite{saxe2014exact}. Notably, recent works have extended the analysis of loss landscapes and training dynamics from linear networks to nonlinear networks; see, e.g., \cite{du2018gradient,liu2022spurious,nguyen2017loss,safran2022effective,sun2020global}.

\subsection{Our Contributions}
In this work, our goal is to completely characterize the loss landscape of the regularized DMF problem \eqref{eq:F}. Our contributions are multifold. First, we provide a closed-form characterization of all critical points of Problem \eqref{eq:F} (see \Cref{thm:crit}) despite its non-convexity and hierarchical structure. Specifically, we show that each critical point admits a singular value decomposition with shared singular values and matrices across layers. In contrast to prior works that rely on restrictive assumptions—such as requiring the data matrix $\bm{Y}$ to have distinct singular values or assuming that the input and output layers are narrower than the hidden layers \cite{achour2024loss,kawaguchi2016deep}—our result holds without imposing any conditions on the input data or the widths of the network layers.

Second, building on the above characterization, we provide a comprehensive loss landscape analysis of Problem \eqref{eq:F} by classifying all of its critical points. Specifically, we provide precise conditions under which a critical point of Problem \eqref{eq:F} is a local minimizer, a global minimizer, a strict saddle point, or a non-strict saddle point (see \Cref{thm:land} and \Cref{fig2:env}).  Compared to \cite{mehta2021loss,ziyin2022exact}, we not only show that the loss landscape of deep linear networks with regularization exhibits spurious local minima, but also provide precise conditions under which they occur. This result allows us to derive a necessary and sufficient condition on the regularization parameters $\{\lambda_l\}_{l=1}^L$ (see \eqref{eq:wellcase}) under which Problem \eqref{eq:F} has a partially benign landscape (see \Cref{coro:land 1}). More specifically, if this condition holds, each critical point of Problem \eqref{eq:F} is either a local minimizer or a strict saddle point; otherwise, Problem \eqref{eq:F} has a non-strict saddle point. This, together with the fact that gradient-based methods with random initialization almost always escape strict saddle points \cite{daneshmand2018escaping,lee2019first}, explains the convergence of gradient descent to a local minimizer for solving Problem \eqref{eq:F} when the regularization parameters satisfy \eqref{eq:wellcase}.
We also conduct numerical experiments to visualize the loss landscape of Problem \eqref{eq:F} in support of our theory.

\subsection{Notation and Definitions}

Given an integer $n$, we denote by $[n]$ the set $\{1,\dots,n\}$.
Given a vector $\bm{a}$, let $\|\bm{a}\|$ denote its Euclidean norm, $\|\bm a\|_0$ its $\ell_0$-norm, $a_i$ its $i$-th entry, and $\mathrm{diag}(\bm{a})$ the diagonal matrix with $\bm{a}$ as its diagonal. For a matrix $\bm{A}$, we define $\|\bm{A}\|$ as its spectral norm, $\|\bm{A}\|_F$ as its Frobenius norm, $a_{ij}$ or $\bm A(i,j)$ as its $(i,j)$-th entry, and $\sigma_i(\bm{A})$ as its $i$-th largest singular value. We denote by $\bm{0}_m$ the $m$-dimensional all-zero vector, by $\bm{0}_{m \times n}$ the $m \times n$ all-zero matrix, and simply by $\bm{0}$ when dimensions can be inferred from context.  We use $\mathcal{O}^{n \times d}$ to denote the set of all $n \times d$ matrices with orthonormal columns, $\mathcal{O}^{n}$ to denote the set of all $n \times n$ orthogonal matrices, and $\mathcal{P}^{n}$ to denote the set of all $n \times n$ permutation matrices.
We denote by $\mathrm{BlkD}(\bm X_1,\ldots,\bm X_n)$ the block-diagonal matrix with diagonal blocks $\bm X_1,\ldots,\bm X_n$ and conformal zero off-diagonal blocks. Blocks with zero rows or zero columns are allowed and are simply omitted from displayed block matrices. In particular, $\bm X_i\in\mathbb R^{m_i\times 0}$ contributes only $m_i$ rows and $\bm X_i\in\mathbb R^{0\times n_i}$ contributes only $n_i$ columns. We also use the convention $\bm A\bm B=\bm 0_{m\times n}$ for $\bm A\in\mathbb R^{m\times 0}$ and $\bm B\in\mathbb R^{0\times n}$.
Given matrices $\{\bm W_l\}_{l=1}^L$ with $\bm W_l \in \R^{d_l\times d_{l-1}}$ for each $l \in [L]$, let $\bm W_{i:1} := \bm W_i\bm W_{i-1}\cdots\bm W_1$ for each $i=2,\dots,L$ and $\bm W_{L:i} := \bm W_L\bm W_{L-1}\cdots \bm W_i$ for each $i \in [L-1]$. In particular, we define $\bm W_{0:1} := \bm I$ and $\bm W_{L:L+1} := \bm I$. For all $i \ge j+1$, let $\bm W_{i:j}=\bm W_{i}\bm W_{i-1}\cdots\bm W_j$ and  $\bm W_{i:j}=\bm W_i$ when $i=j$. In addition, we define $\prod_{l=L}^1 \bm W_l := \bm W_{L}\cdots\bm W_1$.

Let $\pi : [d] \rightarrow [d]$ denote a permutation of the elements in $[d]$. Note that there is a one-to-one correspondence between permutation matrix $\bm\Pi \in \mathcal{P}^{d}$ and a permutation $\pi$, i.e., $\bm \Pi(i,j) = 1$ if $j = \pi(i)$ and $\bm \Pi(i,j) = 0$ otherwise for each $i \in [d]$. Now, suppose that $\bm\Pi$ corresponds to $\pi$. This implies that $\bm \Pi^T$ corresponds to $\pi^{-1}$. Moreover, we simply write $\nabla_{\bm W_l} F(\bm W)$ by $\nabla_l F(\bm W)$ for each $l \in [L]$. We denote the Hessian matrix of $ F $ at $ \bm{W} $ by $ \nabla^2 F(\bm{W}) $ and  the bilinear form of Hessian along a direction $\bm D$ by
\begin{align}\label{eq:defhessian}
    \nabla^2 F(\bm W)[\bm D, \bm D] &:= \left\langle \bm D, \lim_{t\to 0} \frac{\nabla F(\bm W+ t\bm D) -\nabla F(\bm W)}{t} \right\rangle \notag \\
    & = 2\lim_{t\to 0}\frac{F(\bm W+t\bm D)-F(\bm W)-\langle\nabla F(\bm W),t\bm D\rangle}{t^2}.
\end{align}

Moreover, we formally introduce the definitions of critical points, local/global minimizers, and strict/non-strict saddle points for Problem \eqref{eq:F} as follows.
\begin{definition}\label{def:crit}
    For Problem \eqref{eq:F}, we say that \\
(i) $\bm W^*$ is a global minimizer if $F(\bm W^*) \le F(\bm W)$ for all $\bm W$. \\
(ii)  \( \bm{W}^* \) is a local minimizer (resp. maximizer) if there exists a neighborhood $\cal U$ of $\bm W^*$ such that $F(\bm W^*) \le F(\bm W)$ (resp. $F(\bm W^*) \ge F(\bm W)$) for all $\bm W \in \cal U$. \\
(iii) $\bm W^*$ is a (first-order) critical point if $\nabla F(\bm W^*) = \bm 0$. $\bm W^*$ is a second-order critical point if $\nabla F(\bm W^*)=\bm 0$ and $\nabla^2 F(\bm W^*)[\bm D, \bm D] \ge 0$  for all $\bm D$. \\
(iv)  $\bm W^*$ is a saddle point if it is a first-order critical point and is neither a local minimizer nor a local maximizer. \\
(v) Suppose that $\bm W^*$ is a saddle point. Then, $\bm W^*$ is a strict saddle point if there exists a direction $\bm D$ such that $\nabla^2F(\bm W^*)[\bm D, \bm D] < 0$.
Otherwise, $\bm W^*$ is a non-strict saddle point.
\end{definition}

The rest of our paper is organized as follows.  In \Cref{sec:main}, we present and discuss the main results concerning the loss landscape of Problem \eqref{eq:F}. In \Cref{sec:pf}, we provide the detailed proofs of the main results. Finally, we report experimental results in \Cref{sec:expe} and conclude the paper in \Cref{sec:conc}.

\section{Main Results}\label{sec:main}
Throughout the rest of the paper, let $ d_Y := \min\{d_0,d_L\}$ and $d_{\min} := \min \{d_0,d_1,\dots,d_L\}$. We focus on the non-scalar
case
\(
    \max_{0\le j\le L} d_j\ge2,
\)
since the fully scalar case \(d_0=\cdots=d_L=1\) can be handled separately
by a direct scalar argument. Clearly, \(d_{\min}\le d_Y\). For ease of exposition, we denote the critical point set of Problem \eqref{eq:F} by
\begin{align}\label{set:crit F}
 \mathcal{W}_F := \left\{ \bm W = (\bm W_1,\dots,\bm W_L): \nabla F(\bm W) = \bm 0 \right\}.
\end{align}

\subsection{Characterization of the Critical Point Set}\label{subsec:crit}

To characterize the critical point set $\mathcal{W}_F$ of Problem \eqref{eq:F}, we first introduce a singular value decomposition (SVD) of the data matrix $\bm Y \in \R^{d_L\times d_0}$. Specifically, let
$
    \bm Y = \bm U_Y\bm \Sigma_Y \bm V_Y^T
$
be an SVD of $\bm Y$, where $\bm U_Y \in \mathcal{O}^{d_L}$ and $\bm V_Y \in \mathcal{O}^{d_0}$. In addition,
\begin{align}\label{eq:Y}
\bm \Sigma_Y = \blk \left( \mathrm{diag}(y_1,\dots,y_{d_Y}),\bm 0_{(d_L-d_Y) \times (d_0-d_Y)}\right) \in \R^{d_L\times d_0},
\end{align}
where $y_1\ge y_2 \ge \dots \ge y_{d_Y} \ge 0$ are singular values. In the literature, it is common to assume that the singular values of $\bm Y$ are distinct to simplify the analysis; see, e.g., \cite{achour2024loss,kawaguchi2016deep}. However, this assumption is strict and does not hold in many practical scenarios. For example, when $\bm Y$ is a membership matrix in $K$-classification problems, it typically has $K$ repeated singular values. To address this issue, we introduce the following setup and notions. Let $r_Y$ denote the rank, i.e., $r_Y:=\mathrm{rank}(\bm Y)$, and $p_Y$ denote the number of distinct positive singular values of $\bm Y$. In other words, there exist indices $s_0, s_1, \dots, s_{p_Y}$ such that $0= s_{0} < s_{1} < \dots < s_{p_Y} = r_Y$ and
\vspace{-0.4\baselineskip}
\begin{subequations}\label{eq:SY}
\begin{align}
    & y_{s_{0}+1} = \dots = y_{s_{1}} > y_{s_{1}+1} = \dots = y_{s_{2}} > \dots > y_{s_{p_Y-1}+1} = \dots = y_{s_{p_Y}} > 0,  \\
    & y_{s_{p_Y}+1} = \cdots = y_{d_{Y}} = 0.
\end{align}
\end{subequations}
Then, let $h_i := s_i - s_{i-1}$ denote the multiplicity of the $i$-th largest positive singular value for each $i \in [p_Y]$. Consequently, we have $\sum_{i=1}^{p_Y} h_i = r_Y$.  Throughout this paper, we consistently use the above notations.  Now, we characterize the critical point set $\mathcal{W}_F$ of Problem \eqref{eq:F} as follows:
\begin{theorem}\label{thm:crit}
    It holds that $\bm W \in \mathcal{W}_F$ if and only if
    \begin{subequations}\label{eq:crit sol}
        \begin{align}
    & \bm W_1 = \bm Q_2\bm \Sigma_1 \mathrm{BlkD}\left(\bm \Pi, \bm I \right)\mathrm{BlkD}\left( \bm O_1,\dots,\bm O_{p_{Y}},\bm O_{p_{Y}+1} \right)\bm V_Y^T, \\
    & \bm W_l = \bm Q_{l+1} \bm \Sigma_l \bm Q_l^T,\ l=2,\dots,L-1,  \\
    & \bm W_L =  \bm U_Y\mathrm{BlkD}\left( \bm O_1^T,\dots,\bm O_{p_{Y}}^T,\widehat{\bm O}_{p_{Y}+1}^T \right)\mathrm{BlkD}\left(\bm \Pi^T, \bm I \right)\bm \Sigma_L \bm Q_L^T,
        \end{align}
    \end{subequations}
where (i) $\bm Q_l \in \mathcal{O}^{d_{l-1}}$ for each $l=2,\dots,L$, (ii) $\bm O_i \in \mathcal{O}^{h_i}$ for each $i \in [p_Y]$, $\bm O_{p_{Y}+1} \in \mathcal{O}^{d_0 - r_{Y}}$, and
$\widehat{\bm O}_{p_{Y}+1}  \in \mathcal{O}^{d_L - r_{Y}}$, and (iii) $\bm \Sigma_l = \mathrm{BlkD}\left(\mathrm{diag}( \bm \sigma)/\sqrt{\lambda_l}, \bm 0\right) \in \R^{d_l\times d_{l-1}}$ for each $l \in [L]$ with $(\bm \sigma, \bm \Pi) \in \R^{d_{\min}} \times \mathcal{P}^{d_Y}$ satisfying
\begin{align}\label{eq:thmsimga pi}
    \sigma_i^{2L-1} - \sqrt{\lambda} y_{\pi(i)}\sigma_i^{L-1} + \lambda \sigma_i = 0,\ \forall i \in [d_{\min}],\ \sigma_1 \ge \sigma_2 \ge \dots \ge \sigma_{d_{\min}} \ge 0.
\end{align}
Here, $\pi:[d_{Y}] \to [d_{Y}]$ is a permutation corresponding to $\bm \Pi \in \mathcal{P}^{d_Y}$ and $\lambda := \prod_{l=1}^L \lambda_l$.
\end{theorem}

We now make some remarks on this theorem. Despite the non-convexity and inherent hierarchical structure of Problem \eqref{eq:F}, we derive a closed-form characterization of its critical points. Specifically, each weight matrix at any critical point admits an SVD as given in \eqref{eq:crit sol}. Note that all weight matrices share the same singular values defined in \eqref{eq:thmsimga pi} up to scaling of the regularization parameters $\{\lambda_l\}_{l=1}^L$, and their left and right singular matrices are determined by the orthogonal matrices $\bm Q_l$ and $\bm O_i$. Here, $\bm Q_l$ for all $l$ are introduced to handle the rotational invariance in the sequential matrix multiplication of the weight matrices, while $\bm O_i$ for all $i$ are used to address the repeated singular values in $\bm Y$. This theorem serves as a cornerstone for characterizing the loss landscape of Problem \eqref{eq:F}.

To understand the properties of each critical point of Problem \eqref{eq:F}, we analyze the roots of the equation $f(x;y) = 0$ for different values of $y$ (see \Cref{subsec:function}), where
\begin{align}\label{eq:deffty}
f(x; y) := x^{2L - 1} - \sqrt{\lambda} y x^{L - 1} + \lambda x,\ \text{where}\ x, y\ge 0.
\end{align}
Notably, it has at most 3 nonnegative roots. We classify its non-negative roots into three cases: (i) the equation has only the zero root, (ii) the equation has exactly one positive root in addition to the zero root, and (iii) the equation has two distinct positive roots in addition to the zero root. For ease of exposition, we use $x(y)$ to denote a positive root of the equation $f(x;y) = 0$. If $f(x;y) = 0$ has two distinct positive roots, we denote the larger positive root of $f(x;y) = 0$ as $\overline{x}(y)$ and the smaller one as $\underline{x}(y)$. If $f(x;y) = 0$ has only one positive root, we denote it as $\hat{x}(y)$.\footnote{Repeated roots are considered as a single root.} Then, we classify the positive singular values of critical points of Problem \eqref{eq:F} as follows:
\begin{align}\label{eq:defsetgamma}
    \mathcal{S}_1 =  \bigcup_{i \in [d_Y]}  \left\{ \overline{x}(y_i)\right\},\quad
    \mathcal{S}_2 = \bigcup_{i \in [d_Y]} \left\{ \underline{x}(y_i) \right\},\quad
    \mathcal{S}_3 = \bigcup_{i \in [d_Y]} \left\{ \hat{x}(y_i) \right\}.
\end{align}
Throughout the rest of this paper, we consistently use the above notation to analyze the critical points of Problem \eqref{eq:F}.

\subsection{Characterization of the Loss Landscape}\label{subsec:loss}

\begin{figure}[t]
\centering
\includegraphics[width=0.8\textwidth, keepaspectratio]{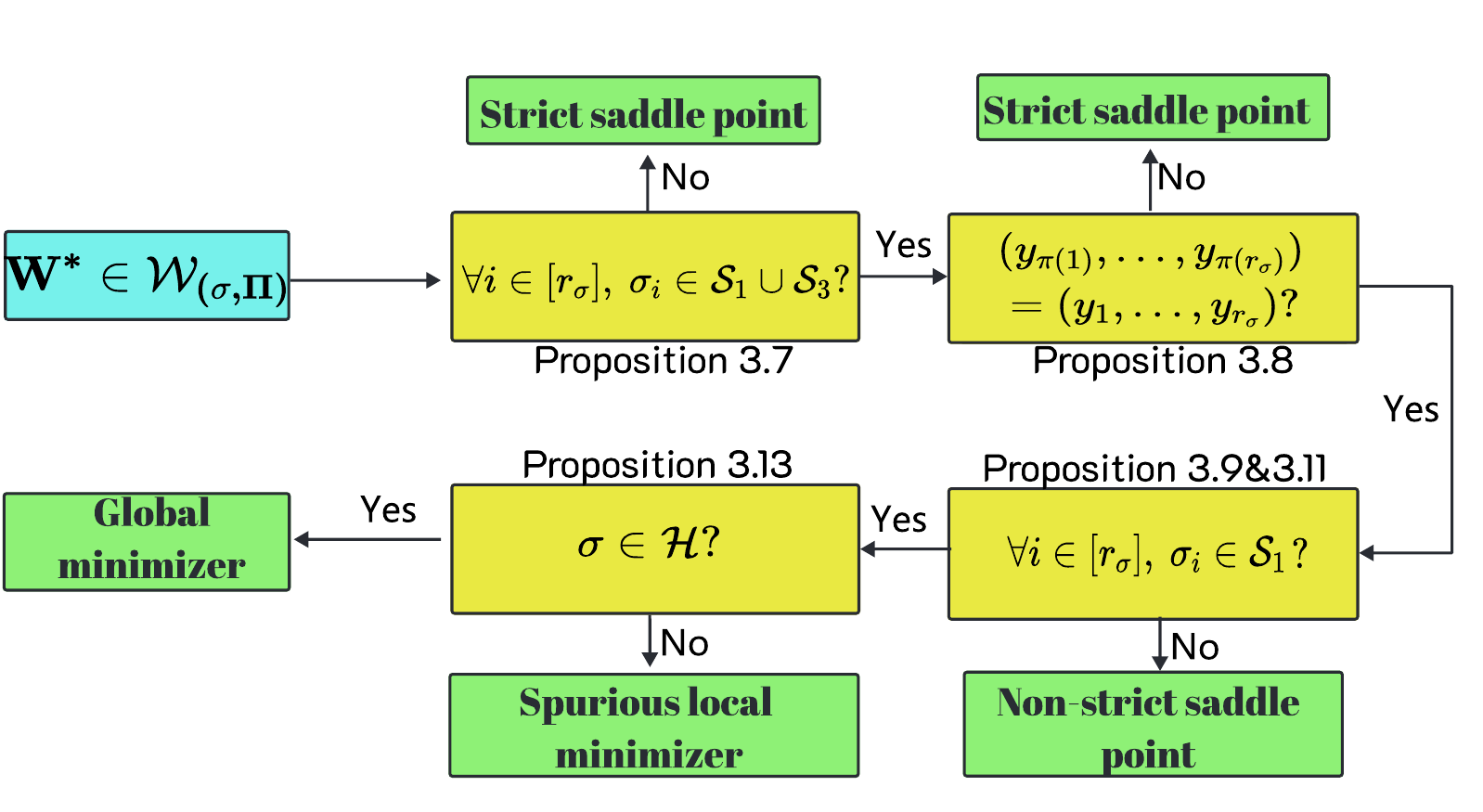}\vspace{-0.1in}
\caption{{\bf Flowchart for characterizing different types of critical points of Problem \eqref{eq:F} when $L \geq 3$}.}
\label{fig2:env}\vspace{-0.25in}
\end{figure}

Before we proceed, we should point out that the work \cite{zhou2022optimization} has shown that when $L=2$, each critical point of Problem \eqref{eq:F} is either a global minimizer or a strict saddle point (see \Cref{sec:app A}). However, when $L\ge 3$, Problem \eqref{eq:F} exhibits a more complex hierarchical structure, which requires a deeper theoretical analysis. Based on the characterization in \Cref{thm:crit}, we now classify all critical points of Problem \eqref{eq:F} into local minimizers, global minimizers, strict saddle points, and non-strict saddle points as follows.

\begin{theorem}\label{thm:land}
    Suppose that $L \ge 3$. Let \( \bm{W}=(\bmw_1,\ldots,\bmw_L)\) be a critical point of Problem \eqref{eq:F} determined by the pair $(\bm{\sigma}, \bm{\Pi}) \in \R^{d_{\min}} \times \mathcal{P}^{d_Y}$ as defined in \Cref{thm:crit}. In addition, let $r_{\sigma} := \|\bm \sigma\|_0$ and $\pi:[d_{Y}] \to [d_{Y}]$ be a permutation corresponding to $\bm \Pi$. Then, the following statements hold: \\
    (i) If there exists \( i \in [r_\sigma] \) such that \( \sigma_i \in \mathcal{S}_2 \) or \( (y_{\pi(1)}, \ldots, y_{\pi(r_\sigma)}) \neq (y_1, \ldots, y_{r_\sigma}) \), then \( \bm{W} \) is a strict saddle point.  \\
    (ii) If \(  \sigma_i \in \mathcal{S}_1 \) for all $ i \in [r_\sigma]$ and \( (y_{\pi(1)}, \ldots, y_{\pi(r_\sigma)}) = (y_1, \ldots, y_{r_\sigma}) \), then \( \bm{W} \) is a local minimizer. \\
    (iii) Suppose that the conditions in (ii) hold. Define $\lambda := \prod_{l=1}^L \lambda_l$ and
    \begin{align}\label{eq:defopt}
        \mathcal{H} := \left\{ \bm \sigma \in \R^{d_{\min}}: \sigma_i \in \mathrm{argmin}_{x \ge 0} (x^L-\sqrt{\lambda}y_i)^2 + \lambda L x^2,\ \forall i \in [d_{\min}] \right\}.
    \end{align}
    If \( \bm{\sigma} \in \mathcal{H} \), \( \bm{W} \) is a global minimizer;  otherwise, \( \bm{W} \) is a spurious local minimizer. \\
    (iv) If there exists \( i \in [r_\sigma] \) such that \( \sigma_i \in \mathcal{S}_3 \) and \( (y_{\pi(1)}, \ldots, y_{\pi(r_\sigma)}) = (y_1, \ldots, y_{r_\sigma}) \), \( \bm{W} \) is a non-strict saddle point.
\end{theorem}
Given any critical point of Problem \eqref{eq:F}, this theorem provides precise conditions under which it is a local minimizer, a global minimizer, a strict saddle point, or a non-strict saddle point. In particular, we illustrate the entire classification pipeline in \Cref{fig2:env}, which offers a complete classification of all critical points of Problem \eqref{eq:F}. We now discuss several implications of this theorem.

First, \Cref{thm:land} shows that Problem \eqref{eq:F} has non-strict saddle points under the condition in (iv).
The key difference from the condition in (ii) is that the condition in (iv) allows a singular value $\sigma_i \in \mathcal{S}_3$.
In \Cref{lem:illcase},
we show that this holds if and only if
\begin{align}\label{eq:illcase}
 \exists i \in [d_Y],\ \prod_{l=1}^L \lambda_l = y_i^{2(L-1)}\left( \left( \frac{L-2}{L} \right)^{\frac{L}{2(L-1)}} + \left(\frac{L}{L-2} \right)^{\frac{(L-2)}{2(L-1)}} \right)^{-2(L-1)}.
\end{align}
This, together with \Cref{thm:land}, yields the following result:
\begin{corollary}\label{coro:land 1}
Suppose $L \ge 3$.
If
\begin{align}\label{eq:wellcase}
\prod_{l=1}^L \lambda_l \neq y_i^{2(L-1)}
\left(
\left( \frac{L-2}{L} \right)^{\frac{L}{2(L-1)}} + \left(\frac{L}{L-2} \right)^{\frac{L-2}{2(L-1)}}
\right)^{-2(L-1)},\quad \forall i \in [d_{\min}],
\end{align}
every critical point of Problem \eqref{eq:F} is either a local minimizer or a strict saddle point; otherwise, there exists a non-strict saddle point.
\end{corollary}
Notably, this corollary establishes a necessary and sufficient condition under which all critical points of Problem~\eqref{eq:F} only consist of local minimizers and strict saddle points. This, together with the fact that gradient-based methods with random initialization almost always escape strict saddle points \cite{daneshmand2018escaping,lee2019first}, implies that gradient descent almost always converges to a local minimizer of Problem \eqref{eq:F} when the regularization parameters satisfy \eqref{eq:wellcase}. Moreover, the work \cite{chen2025error} has shown that under the condition  \eqref{eq:wellcase} and a mild condition on layer widths, Problem \eqref{eq:F} satisfies the error bound condition, which can be used to establish linear convergence of first-order methods \cite{Luo1993,zhou2017unified}. Combining all these results together yields that gradient descent with random initialization almost always converges to a local minimizer at a linear rate. On the other hand, if the regularization parameters $\{\lambda_l\}_{l=1}^L$ do not satisfy \eqref{eq:wellcase}, Problem \eqref{eq:F} possesses a non-strict saddle point. Consequently, gradient-based methods may converge to these non-strict saddle points rather than a local minimizer, which is usually not desired in practice \cite{achour2024loss}.

Second, \Cref{thm:land} demonstrates that Problem \eqref{eq:F} has spurious local minimizers when $L \ge 3$, i.e., not all local minimizers are global minimizers. This, together with the previous findings that no spurious local minima exist when $L=2$ (see, e.g., \cite{bhojanapalli2016global,ge2017no,li2019symmetry,zhou2022optimization}), reveals a fundamental difference in the loss landscape between shallow and deep models. Moreover, our result is consistent with recent findings (e.g., \cite{hepiecewise,liu2022spurious}), which show that deep neural networks with piecewise linear activations often exhibit spurious local minima.

Finally, Problem \eqref{eq:F} can be viewed as an instance of the regularized mean-squared loss for deep linear networks with an orthogonal input \cite{arora2019convergence,chen2025error}. The recent work \cite{achour2024loss} presented a comprehensive analysis of the loss landscape for deep linear networks. Here, we highlight the key differences between their results and ours. First, we focused on the regularized deep matrix factorization problem \eqref{eq:F}, while they focused on unregularized deep linear networks with square loss $\min_{\bm W} \left\| \bm W_L\cdots\bm W_1\bm X - \bm Y\right\|_F^2$. Notably, the product $\bm W_L\cdots\bm W_1$ in this problem can be viewed as a single matrix to simplify analysis. However, such a simplification is not possible for Problem \eqref{eq:F} due to the presence of explicit regularization terms applied to each layer. Second, the work \cite{achour2024loss} showed that all critical points of the unregularized problem consist of global minimizers, strict saddle points, and non-strict saddle points. In contrast, our result in \Cref{coro:land 1} demonstrates that as long as the regularization parameters satisfy \eqref{eq:wellcase}, Problem \eqref{eq:F} does not have any non-strict saddle points. To sum up, studying Problem \eqref{eq:F} not only facilitates our understanding of DMF but also offers better insights into the optimization properties of deep linear models with regularization. {  We should point out that \cite{achour2024loss} can handle non-orthogonal data thanks to the absence of regularization, while our analysis is limited to orthogonal input data.}

\section{Proof of Main Results}\label{sec:pf}

To facilitate our analysis, we introduce the following auxiliary problem associated with Problem \eqref{eq:F}:
\begin{align}\label{eq:G}
    \min_{\bm W}\ G(\bm W) := \left\|\bm W_L\cdots\bm W_1 -  \sqrt{\lambda}\bm Y \right\|_F^2 + \lambda \sum_{l=1}^L \|\bm W_l\|_F^2,
\end{align}
where $\lambda := \prod_{l=1}^L \lambda_l$. We compute the gradient of $G(\bm W)$ as follows:
\begin{align}\label{eq:grad G}
\nabla_{l} G(\bm W) = 2\bm W_{L:l+1}^T\left(\bm W_{L:1} - \sqrt{\lambda}\bm Y\right)\bm W_{l-1:1}^T + 2\lambda \bm W_l,\ \forall l \in [L].
\end{align}
The critical point set of Problem \eqref{eq:G} is defined as
\begin{align}\label{set:crit G}
\mathcal{W}_G := \left\{ \bm W = (\bm W_1,\dots,\bm W_L): \nabla  G(\bm W) = \bm 0 \right\}.
\end{align}
Now, we present the following lemma to establish the equivalence between Problems \eqref{eq:F} and \eqref{eq:G} in terms of their critical point sets and loss landscapes.

\begin{lemma}\label{lem:equi FG}
Consider Problems \eqref{eq:F} and \eqref{eq:G}. The following statements hold: \\
(i) $(\bm W_1,\dots,\bm W_L) \in \mathcal{W}_F$ if and only if $(\sqrt{\lambda_1}\bm W_1,\dots,\sqrt{\lambda_L}\bm W_L) \in \mathcal{W}_G$. \\
(ii) $(\bm W_1,\dots,\bm W_L)$ is a local minimizer/maximizer, global minimizer, strict saddle point, or non-strict saddle point of Problem \eqref{eq:F} if and only if $(\sqrt{\lambda_1}\bm{W}_1, \ldots, \sqrt{\lambda_L}\bm{W}_L)$ is the same type of critical point of Problem \eqref{eq:G}. \\
(iii) Let $\bm{Y} = \bm{U}_Y \bm{\Sigma}_Y \bm{V}_Y^T$ be an SVD of $\bm{Y}$, where $\bm U_Y \in \mathcal{O}^{d_L}$, $\bm \Sigma_Y \in \R^{d_L\times d_0}$, and $\bm V_Y \in \mathcal{O}^{d_0}$. It holds that $(\bm W_1,\bm W_2,\dots,\bm W_{L-1},\bm W_L)$ is a critical point of Problem \eqref{eq:G} if and only if  $(\bm W_1\bm V_Y,\bm W_2,\dots,\bm W_{L-1},\bm U_Y^T\bm W_L)$ is a critical point of the problem $\min_{\bm W}\ \left\|\bm W_L\cdots\bm W_1 -  \sqrt{\lambda}\bm \Sigma_Y \right\|_F^2 + \lambda \sum_{l=1}^L \|\bm W_l\|_F^2$.
\end{lemma}
\begin{proof}
(i) Let $\bm W\!\!:=\! (\bm W_1,\dots,\bm W_L)\!\in\! \mathcal{W}_F$ be arbitrary. We have
\begin{align}
\nabla_l F(\bm W) = 2\bm W_{L:l+1}^T\left(\bm W_{L:1} - \bm Y\right)\bm W_{l-1:1}^T + 2 \lambda_l \bm W_l,\ \forall l \in [L].
\notag
\end{align}
Now, let $\hat{\bm W} := (\sqrt{\lambda_1}\bm W_1,\dots,\sqrt{\lambda_L}\bm W_L)$. For each $l\in [L]$, we compute
\begin{align*}
\nabla_{l} G(\hat{\bm W}) &\overset{\eqref{eq:grad G}}{=}  2\sqrt{\lambda_L\cdots\lambda_{l+1}\lambda_{l-1}\cdots\lambda_1}{\bm W_{L:l+1}^T}(\sqrt{\lambda} \bm W_{L:1} - \sqrt{\lambda}\bm Y)\bm W_{l-1:1}^T   + 2\lambda\sqrt{\lambda_l} \bm W_l \\
&= \tfrac{\lambda}{\sqrt{\lambda_l}} \nabla_l F(\bm W)   .
\end{align*}
Therefore, $\nabla_l F(\bm W)=\bm 0$ if and only if $\nabla_{l} G(\hat{\bm W}) =\bm 0$. This yields the desired result.

(ii) Let $\varphi(\bm W) := (\sqrt{\lambda_1}\bm W_1,\dots,\sqrt{\lambda_L}\bm W_L)$. Note that $G(\varphi(\bm{W})) = \lambda F(\bm{W})$ and that $\varphi(\bm{W})$ is a smooth bijection in $\bm{W}$. By \Cref{def:crit} and \Cref{lem:boumalcite}, this yields the desired result.

(iii)
Combining the SVD of $\bm Y$ and \eqref{eq:grad G} directly yields the desired result.
\end{proof}
As established above, the characterization of the critical point set and loss landscape of Problem~\eqref{eq:F} is equivalent to that of Problem~\eqref{eq:G}. By \Cref{lem:equi FG}(iii), we assume without loss of generality that $\bm{Y}= \bm{\Sigma}_Y$ throughout the rest of the paper, i.e., where $\bm{\Sigma}_Y$ is defined in \eqref{eq:Y}.

\subsection{Analysis of the Critical Point Set}\label{subsec:pre}

In this subsection, we analyze the critical point set $\mathcal{W}_G$ of Problem \eqref{eq:G}.
To begin, we present a lemma to show that the weight matrices $\{\bm W_l\}_{l=1}^L$ at any critical point of Problem \eqref{eq:G} are balanced.
\begin{lemma}\label{lem:bala}
Let $(\bm W_1,\dots,\bm W_L)$ be a critical point of Problem \eqref{eq:G}. We have
\begin{align}
& \bm W_l\bm W_l^T = \bm W_{l+1}^T\bm W_{l+1},\ \forall l \in [L-1],\quad \text{and}\ \label{eq:bala} \\
& (\bm W_l \bm W_l^T)^{L-1} \bm W_l - \sqrt{\lambda}\bm W_{L:l+1}^T \bm Y \bm W_{l-1:1}^T + \lambda \bm W_l=0, \ \forall l \in [L].\label{eq:crit}
\end{align}
\end{lemma}
\begin{proof}
According to $\nabla_l G(\bm W) = \bm 0$ for each $l \in [L]$, we have $\nabla_l G(\bm W)\bm W_l^T - \bm W_{l+1}^T \nabla_{l+1} G(\bm W)= \bm 0$ for all $l \in [L-1]$. This, together with \eqref{eq:grad G}, implies \eqref{eq:bala}.
Recursively using \eqref{eq:bala}, we have
\begin{align*}
& \bm W_{L:l+1}^T \bm W_{L:1}\bm W_{l-1:1}^T  = \bm W_{l+1}^T\cdots\bm W_{L-1}^T\bm W_L^T\bm W_L\bm W_{L-1}\cdots\bm W_2\bm W_1\bm W_1^T\bm W_2^T\cdots\bm W_{l-1}^T \\
= &\ \ \bm W_{l+1}^T\cdots\left(\bm W_{L-1}^T\bm W_{L-1}\right)^2\cdots\left(\bm W_2\bm W_2^T\right)^2 \cdots\bm W_{l-1}^T = (\bm W_l \bm W_l^T)^{L-1} \bm W_l.
\end{align*}
Substituting this into \eqref{eq:grad G}, together with $\nabla_l G(\bm W) = \bm 0$, yields \eqref{eq:crit}.
\end{proof}

With the above setup, we present a closed-form characterization of $\mathcal{W}_G$ as follows:

\begin{proposition}\label{prop:opti G}
Suppose that $\bm Y \in \R^{d_L\times d_0}$ takes the form of \eqref{eq:Y}. The critical point set of Problem \eqref{eq:G} can be expressed as
\begin{equation*}
    \mathcal{W}_G = \bigcup_{(\bm{\sigma}, \bm{\Pi}) \in \mathcal{B}} \mathcal{W}_{(\bm{\sigma}, \bm{\Pi})},
\end{equation*}
where $\mathcal{W}_{(\bm\sigma,\bm\Pi)}$ takes the form of
\begin{align}\label{set:sigmapi}
& \left\{\bm W :
\begin{array}{l}
\bm \Sigma_l = \mathrm{BlkD}\left(\mathrm{diag}( \bm \sigma), \bm 0\right) \in \R^{d_l\times d_{l-1}},\ \forall l \in [L],\\
\bm W_1 = \bm Q_2\bm \Sigma_1 \mathrm{BlkD}\left(\bm \Pi, \bm I  \right)\mathrm{BlkD}\left( \bm O_1,\dots,\bm O_{p_{Y}},\bm O_{p_{Y}+1} \right), \\
\bm W_l = \bm Q_{l+1} \bm \Sigma_l \bm Q_l^T,\ l=2,\dots,L-1,\ \bm Q_l \in \mathcal{O}^{d_{l-1}},\ l=2,\dots,L,  \\
\bm W_L =  \mathrm{BlkD}\left( \bm O_1^T,\dots,\bm O_{p_{Y}}^T,\widehat{\bm O}_{p_{Y}+1}^T \right)\mathrm{BlkD}\left(\bm \Pi^T, \bm I \right)\bm \Sigma_L \bm Q_L^T, \\
\bm O_i \in \mathcal{O}^{h_i},\ \forall i \in [p_{Y}],\  \bm O_{p_{Y}+1} \in \mathcal{O}^{d_0 - r_{Y}},\ \widehat{\bm O}_{p_{Y}+1}  \in \mathcal{O}^{d_L - r_{Y}},
\end{array}
\right\}
\end{align}
and
\begin{align*}
\mathcal{A} & := \left\{ \bm a \in \R^{d_{Y}}: \|\bm a\|_0 \le d_{\min},\ a_i^{2L-1} - \sqrt{\lambda}y_ia_i^{L-1} + \lambda a_i = 0,\ a_i \geq 0,\ \forall i \in [d_{Y}] \right\}, \notag \\
\mathcal{B} & := \left\{ (\bm \sigma, \bm \Pi)  \in \R^{d_{\min}} \times \mathcal{P}^{d_{Y}}:
\bm a \in \mathcal{A},\ \left( \bm \sigma, \bm 0_{d_Y-d_{\min}} \right) =\bm \Pi \bm a,\ \sigma_1 \geq \dots \geq \sigma_{d_{\min}}  \right\}. \notag
\end{align*}
\end{proposition}
\begin{proof}
Recall that
\begin{align}\label{eq:Y1}
\bm Y = \mathrm{BlkD}\left(\tilde{\bm \Sigma}_Y,\bm 0_{(d_L-d_Y) \times (d_0-d_Y)}\right),\ \text{where}\ \tilde{\bm \Sigma}_Y := \mathrm{diag}(y_1,\dots,y_{d_{Y}}).
\end{align}
Let $(\bm \sigma, \bm \Pi)\in \mathcal{B}$ and $\bm W \in \mathcal{W}_{(\bm \sigma,\bm \Pi)}$ be arbitrary. We first verify that $\nabla G(\bm W) = \bm 0$. For any $\bm a \in \cal A$, it follows from $\tilde{\bm \Sigma}_Y = \mathrm{diag}(y_1,\dots,y_{d_{Y}})$ that
\begin{align}\label{eq0:prop opti G}
     \mathrm{diag}^{2L-1}(\bm a) - \sqrt{\lambda} \mathrm{diag}^{L-1}(\bm a) \tilde{\bm \Sigma}_{Y} + \lambda \mathrm{diag}(\bm a)  = \bm 0.
\end{align}
Since $(\bm \sigma, \bm \Pi) \in \cal B$, there exists $\bm a\in\mathcal A$ such that
$$\bm \Pi^T\blk\left(\mathrm{diag}(\bm \sigma), \bm 0_{(d_Y-d_{\min})\times (d_Y-d_{\min})}\right) \bm \Pi = \mathrm{diag}(\bm a).$$ Substituting this into \eqref{eq0:prop opti G} and multiplying both sides by $\bm \Pi$ and $\bm \Pi^T$ yield
\begin{align}\label{eq1:prop opti G}
\begin{bmatrix}
    \mathrm{diag}^{2L-1}(\bm \sigma) & \bm 0 \\
    \bm 0 & \bm 0
\end{bmatrix} - \sqrt{\lambda}\begin{bmatrix}
    \mathrm{diag}^{L-1}(\bm \sigma) & \bm 0 \\
    \bm 0 & \bm 0
\end{bmatrix} \bm \Pi \tilde{\bm \Sigma}_{Y}\bm \Pi^T + \lambda \begin{bmatrix}
    \mathrm{diag}(\bm \sigma) & \bm 0 \\
    \bm 0 & \bm 0
\end{bmatrix} = \bm 0.
\end{align}
For each $l =2,\dots,L-1$, substituting the form of $\{\bm W_l\}_{l=1}^{L}$ in \eqref{set:sigmapi} into \eqref{eq:grad G} yields
\begin{align*}
\frac{1}{2}\nabla_{l}G(\bm W)
& = \bm Q_{l+1}\left( \bm \Sigma_l^T\bm \Sigma_l \right)^{L-1}\bm \Sigma_l\bm Q_l^T + \lambda \bm Q_{l+1} \bm \Sigma_l \bm Q_l^T - \sqrt{\lambda}\bm Q_{l+1} \\
&\quad   \left(\prod_{j=l+1}^L \bm \Sigma^T_j\right) \mathrm{BlkD}\left(\bm
 \Pi, \bm I_{d_L-d_{Y}}\right)\bm Y \mathrm{BlkD} \left(\bm
 \Pi^T, \bm I_{d_0-d_{Y}}\right)\left(\prod_{j=1}^{l-1} \bm \Sigma^T_j\right) \bm Q_{l}^T \\
& = \bm Q_{l+1}  \left( \begin{bmatrix}
    \mathrm{diag}^{2L-1}(\bm \sigma) & \bm 0 \\
    \bm 0 & \bm 0
\end{bmatrix} - \sqrt{\lambda}\begin{bmatrix}
    \mathrm{BlkD}\left(\mathrm{diag}^{L-1}(\bm \sigma),\bm 0\right)\bm \Pi \tilde{\bm \Sigma}_{Y}\bm \Pi^T & \bm 0 \\
    \bm 0 & \bm 0
\end{bmatrix}  \right. \\
 &\quad \left. +\ \lambda \begin{bmatrix}
    \mathrm{diag}(\bm \sigma) & \bm 0 \\
    \bm 0 & \bm 0
\end{bmatrix} \right) \bm Q_l^T \overset{\eqref{eq1:prop opti G}}{=} \bm 0,
  \end{align*}
where the second equality follows from that $\bm \Sigma_l = \mathrm{BlkD}\left(\mathrm{diag}(\bm \sigma), \bm 0\right)$ for all $l\in [L]$ and $\bm \Pi \tilde{\bm \Sigma}_{Y}\bm \Pi^T$ are diagonal matrices.
{
For the case $l=1$,  substituting the form of $\{\bm W_l\}_{l=1}^{L}$ in \eqref{set:sigmapi} into \eqref{eq:grad G} yields
{\small\begin{align*}
     \frac{1}{2}\nabla_{1} G(\bm W) &=  \bm Q_2 \left( (\bm \Sigma_1\bm \Sigma_1^T)^{L-1}\bm \Sigma_1 + \lambda \bm \Sigma_1 \right)\mathrm{BlkD}\left(\bm \Pi, \bm I_{d_0-d_{Y}} \right)\mathrm{BlkD}\left( \bm O_1,\dots,\bm O_{p_Y+1} \right)\\
    &\quad- \sqrt{\lambda} \bm Q_2 \left(\prod_{l=2}^{L}\bm \Sigma_l\right) \mathrm{BlkD}\left(\bm \Pi, \bm I_{d_L - d_{Y}} \right)\mathrm{BlkD}\left( \bm O_1,\dots,\bm O_{p_Y}, \widehat{\bm O}_{p_Y+1} \right)\bm Y \\
    & = \bm Q_2 \left( (\bm \Sigma_1\bm \Sigma_1^T)^{L-1}\bm \Sigma_1 + \lambda \bm \Sigma_1 - \sqrt{\lambda} \left(\prod_{l=2}^{L}\bm \Sigma_l\right) \mathrm{BlkD}\left(\bm \Pi, \bm I \right)  \bm Y \mathrm{BlkD}\left(\bm \Pi^T, \bm I  \right)\right) \\
    &\quad \mathrm{BlkD}\left(\bm \Pi, \bm I \right) \mathrm{BlkD}\left( \bm O_1,\dots,\bm O_{p_Y},\bm O_{p_Y+1} \right) = \bm 0,
\end{align*} }
where the second equality follows from the block structure of \( \bm Y \) in \eqref{eq:Y1} and thus
\[
\mathrm{BlkD}( \bm O_1, \dots, \bm O_{p_Y}, \widehat{\bm O}_{p_Y+1} )\bm Y
= \bm Y\, \mathrm{BlkD}\left( \bm O_1, \dots, \bm O_{p_Y}, \bm O_{p_Y+1} \right),
\]
and the last equality follows from  \( \bm \Sigma_1 = \mathrm{BlkD}\left( \mathrm{diag}(\bm \sigma), \bm 0 \right) \) and \eqref{eq1:prop opti G}. For the case $l=L$, substituting the form of $\{\bm W_l\}_{l=1}^{L}$ in \eqref{set:sigmapi} into \eqref{eq:grad G} yields
{\small\begin{align*}
    \frac{1}{2}\nabla_L G(\bmw)
    &=\mathrm{BlkD}\left( \bm O_1^T,\dots,\bm O_{p_{Y}}^T,\widehat{\bm O}_{p_{Y}+1}^T \right)\mathrm{BlkD}\left(\bm \Pi^T, \bm I \right)((\bm \Sigma_L\bm \Sigma_L^T)^{L-1}\bm \Sigma_L+\lambda\bm \Sigma_L)\bm Q^{T}_L\\
    &\quad-\sqrt{\lambda}\bm Y\mathrm{BlkD}\left( \bm O_1^T,\dots,\bm O_{p_{Y}}^T,{\bm O}_{p_{Y}+1}^T \right)\mathrm{BlkD}\left(\bm \Pi^T, \bm I \right)\bm \Sigma_{L-1:1}^T \bm Q_L^T\\
    &=\mathrm{BlkD}\left( \bm O_1^T,\dots,\bm O_{p_{Y}}^T,\widehat{\bm O}_{p_{Y}+1}^T \right)\mathrm{BlkD}\left(\bm \Pi^T, \bm I \right)
    \big( (\bm \Sigma_L\bm \Sigma_L^T)^{L-1}\bm \Sigma_L+\lambda\bm \Sigma_L   \\
 &\quad -\left. \sqrt{\lambda}\blk\left(\bm \Pi \tilde{\bm \Sigma}_{Y}\bm \Pi^T\mathrm{diag}^{L-1}(\bm \sigma),\bm 0\right)\right)\bm Q_L^T= \bm 0,
\end{align*}}
where the second equality follows from the block structure $\bm{Y}$ in \eqref{eq:Y1} and
\begin{align*}
    &\quad \bm Y \mathrm{BlkD}\left( \bm O_1^T,\dots,\bm O_{p_{Y}}^T,{\bm O}_{p_{Y}+1}^T \right)\mathrm{BlkD}\left(\bm \Pi^T, \bm I \right)\\
    &=\mathrm{BlkD}\left( \bm O_1^T,\dots,\bm O_{p_{Y}}^T,\widehat{\bm O}_{p_{Y}+1}^T \right)\mathrm{BlkD}\left(\bm \Pi^T, \bm I \right)
    \blk(\bm \Pi \tilde{\bm \Sigma}_{Y}\bm \Pi^T , \bm 0).
\end{align*}
Consequently, we conclude that $(\bm W_1,\dots,\bm W_L)$ is a critical point of Problem \eqref{eq:G}.}

Conversely, suppose that $\bm W$ is a critical point, i.e., $\nabla_l G(\bm W) = \bm 0$ for all $l \in [L]$. According to \Cref{lem:bala}, we obtain that $\bm W_l$ for all $l \in [L]$ share the same rank denoted by $r$, which satisfies $r \leq d_{\min}$. For each $l \in [L]$, let
\begin{equation}\label{eq:svdwl}
    \bm W_l = \bm U_l \bm \Sigma_l \bm V_l^T
\end{equation}
be an SVD of $\bm W_l$, where $\bm U_l \in \mathcal{O}^{d_l}$, $\bm V_l \in\mathcal{O}^ {d_{l-1}}$, and $\bm \Sigma_l = \mathrm{BlkD}(\tilde{\bm \Sigma}_l, \bm 0)\in \mathbb{R}^{d_l \times d_{l-1}}$ with $\tilde{\bm \Sigma}_l \in \R^{r\times r}$ being a diagonal matrix with positive diagonal entries. This, together with \eqref{eq:bala}, yields
\begin{align}\label{eq1:bala}
 \bm U_l\bm{\Sigma}_l\bm{\Sigma}_l^T\bm U_l^T = \bm V_{l+1} \bm{\Sigma}_{l+1}^T \bm{\Sigma}_{l+1}\bm V_{l+1}^T,\ \forall l \in [L-1].
\end{align}
Since the above both sides are eigenvalue decompositions of the same matrix, with eigenvalues in decreasing order, we have $ \bm{\Sigma}_l\bm{\Sigma}_l^T =  \bm{\Sigma}_{l+1}^T \bm{\Sigma}_{l+1}$ for each $l \in [L-1]$. This implies that $\bm W_1,\dots,\bm W_L$ have the same positive singular values. Next, let $\{\sigma_i\}_{i=1}^r$ denote the positive singular values of $\bm W_l$ for each $l \in [L]$ and $p$ denote the number of distinct elements of positive singular values. In other words, there exist indices $\hat{s}_0, \hat{s}_1,\dots, \hat{s}_p$ such that $0 = \hat{s}_0 < \hat{s}_1 <\dots < \hat{s}_p=r$ and
\begin{align*}
\sigma_{\hat{s}_0+1} = \dots = \sigma_{\hat{s}_1} > \sigma_{\hat{s}_1+1} = \dots = \sigma_{\hat{s}_2} > \dots > \sigma_{\hat{s}_{p-1}+1} = \dots = \sigma_{\hat{s}_p} > 0.
\end{align*}
Let $\hat{h}_i := \hat{s}_i - \hat{s}_{i-1}$ be the multiplicity of the $i$-th largest positive value  for each $i \in [p]$. With an abuse of notation, we define $\hat{h}_{p+1} := d_{l} - r$ for $l\in [L-1]$.\footnote{More precisely, the last block size depends on the layer and can be written as $\hat h_{p+1,l}:=d_l-r$. With a slight abuse of notation, we omit the layer index $l$ when it is clear from context in this proof.} Then, we have
\begin{align}\label{eq3:bala}
    \tilde{\bm \Sigma}_l = \tilde{\bm \Sigma} := \mathrm{BlkD}\left( \sigma_{\hat{s}_1}  \bm I_{\hat{h}_1},\dots, \sigma_{\hat{s}_p} \bm I_{\hat{h}_p}\right)  \in \R^{r \times r}.
\end{align}
Based on the above block form, we write $\bm U_l$ and $\bm V_l$ in \eqref{eq:svdwl} for each $l \in [L]$ as
\begin{align}\label{eq:Ul Vl}
\bU_{l} =  \left[\bm{U}_{l}^{(1)},\dots,\bm{U}_{l}^{(p)},\bm{U}_{l}^{(p+1)}\right],\ \bV_{l} =  \left[\bm{V}_{l}^{(1)},\dots,\bm{V}_{l}^{(p)},\bm{V}_{l}^{(p+1)}\right],
\end{align}
where $\bm{U}_l^{(i)} \in \mO^{d_l \times \hat{h}_i}$ and $\bm{V}_l^{(i)} \in \mO^{d_{l-1} \times \hat{h}_i}$ for all $i \in [p]$, $\bm{U}_l^{(p+1)} \in \mO^{d_l \times (d_{l}- r)}$, and $\bm{V}_l^{(p+1)} \in \mO^{d_{l-1} \times (d_{l-1}- r)}$. This, together with \eqref{eq1:bala}, \eqref{eq3:bala}, and \cite[Lemma 8(i)]{wang2023understanding}, implies that there exists orthogonal matrix $\bm Q_{l}^{(i)} \in \mO^{\hat{h}_i}$ such that
\begin{equation}\label{eq:u1v1}
\bm{U}_{l}^{(i)} = \bm{V}_{l+1}^{(i)}\bm Q_{l}^{(i)},\ \forall l \in [L-1],\ i \in [p+1].
\end{equation}
Using \eqref{eq:svdwl},
we compute
{
\begin{align}\label{eq1:thm opti G}
\bm W_{L:1}  =\quad ~ &\bm U_L \bm \Sigma_L (\bm V_L^T\bm U_{L-1})\bm \Sigma_{L-1}(\bm V_{L-1}^T\bm U_{L-2})\bm \Sigma_{L-2}\cdots(\bm V_2^T \bm U_1)\bm \Sigma_{1}\bm V_1^T \notag \\
 \overset{(\ref{eq3:bala},\ref{eq:u1v1})}{=} &\bm U_L \blk\left(
        \prod_{l=L-1}^1 \sigma_{\hat{s}_1} \bm Q_{l}^{(1)},\ldots,\prod_{l=L-1}^1 \sigma_{\hat{s}_p} \bm Q_{l}^{(p)},\bm 0 \right)\bm \Sigma_1 \bm V_1^T  \notag  \\
 =\quad ~ &\bm U_L \blk(\widetilde{\bm \Sigma}^L,\bm 0\bm )\bm Q\bm V_1^T,
\end{align}}
where $\prod_{l=L-1}^1 \bm Q_{l}^{(j)} := \bm Q_{L-1}^{(j)} \ldots \bm Q_{1}^{(j)}$ for each $j \in [p]$ and
\begin{equation}\label{eq:Qexpression}
    \bm Q = \mathrm{BlkD}\left(\tilde{\bm Q}, \bm{I}_{d_0-r}\right) = \mathrm{BlkD} \left(\prod_{l=L-1}^1 \bm Q_{l}^{(1)},\dots,\prod_{l=L-1}^1 \bm Q_{l}^{(p)},\bm{I} \right) \in \mathcal{O}^{d_0}.
\end{equation}
Right-multiplying \eqref{eq:crit} by $\bm W_L^T$ when $l=L$ and left-multiplying \eqref{eq:crit} by  $\bm W_1^T$  when $l=1$, we obtain
\begin{align*}
\left(\bm W_L\bm W_L^T \right)^L - \sqrt{\lambda}\bm Y\bm W_{L:1}^T+\lambda\bm W_L\bm W_L^T = \bm 0, \left(\bm W_1^T\bm W_1 \right)^L - \sqrt{\lambda}\bm W_{L:1}^T\bm Y + \lambda\bm W_1^T\bm W_1 = \bm 0.
\end{align*}
Substituting \eqref{eq:svdwl}, \eqref{eq3:bala}, and \eqref{eq1:thm opti G} into the above equations, together with $\bm U_L \in \mathcal{O}^{d_L}$ and $\bm V_1 \in \mathcal{O}^{d_0}$, yields
\begin{align}
& \sqrt{\lambda} \bm U_L^T \bm Y \bm V_1\bm Q^T\mathrm{BlkD}\left(\tilde{\bm \Sigma}^L, \bm 0_{(d_0-r)\times (d_L-r)} \right) =  \mathrm{BlkD}\left(\tilde{\bm \Sigma}^{2L} + \lambda\tilde{\bm \Sigma}^{2}, \bm 0 \right), \label{eq2:thm opti G}\\
& \sqrt{\lambda}\bm Q^T\mathrm{BlkD}\left(\tilde{\bm \Sigma}^L, \bm 0_{(d_0-r)\times (d_L-r)} \right) \bm U_L^T\bm Y\bm V_1  =  \mathrm{BlkD}\left(\tilde{\bm \Sigma}^{2L} + \lambda\tilde{\bm \Sigma}^{2}, \bm 0  \right). \notag
\end{align}
According to the above equation and $\bm Q \in \mathcal{O}^{d_0}$, we obtain
\begin{align*}
\sqrt{\lambda} \mathrm{BlkD}\left(\tilde{\bm \Sigma}^L, \bm 0_{(d_0-r)\times (d_L-r)} \right) \bm U_L^T\bm Y\bm V_1 & =  \bm Q \mathrm{BlkD}\left(\tilde{\bm \Sigma}^{2L} + \lambda\tilde{\bm \Sigma}^{2}, \bm 0_{(d_0-r)\times(d_0-r)} \right)\\
& \overset{(\ref{eq3:bala}, \ref{eq:Qexpression})}{=}  \mathrm{BlkD}\left(\tilde{\bm \Sigma}^{2L} + \lambda\tilde{\bm \Sigma}^{2}, \bm 0 \right)\bm Q.
\end{align*}
Right-multiplying on both sides of the above equality by $\bm Q^T$ yields
\begin{align*}
\sqrt{\lambda}\mathrm{BlkD}\left(\tilde{\bm \Sigma}^L, \bm 0_{(d_0-r)\times (d_L-r)} \right) \bm U_L^T\bm Y\bm V_1\bm Q^T =  \mathrm{BlkD}\left(\tilde{\bm \Sigma}^{2L} + \lambda\tilde{\bm \Sigma}^{2}, \bm 0_{(d_0-r)\times(d_0-r)} \right).
\end{align*}
We now partition $\bm C := \bm U_L^T\bm Y\bm V_1\bm Q^T \in \R^{d_L\times d_0}$ into the block form {\small $\bm C = \begin{bmatrix}
\bm C_1 & \bm C_2\\
\bm C_3 & \bm S
\end{bmatrix}$}, where $\bm C_1 \in \R^{r\times r}$.
This, together with \eqref{eq2:thm opti G} and the above equation, yields
$
\bm C_1 =   ( \tilde{\bm \Sigma}^{L} + \lambda\tilde{\bm \Sigma}^{2-L} )/\sqrt{\lambda},\ \bm C_2 = \bm 0,\ \bm C_3 = \bm 0.
$
Consequently, we obtain
\begin{align}\label{eq5:thm opti G}
\bm U_L^T\bm Y\bm V_1\bm Q^T = \mathrm{BlkD}\left( {\tfrac{1}{\sqrt{\lambda}}} \left( \tilde{\bm \Sigma}^{L} + \lambda\tilde{\bm \Sigma}^{2-L} \right), \bm S\right).
\end{align}
Now, let $\bm U_S\bm \Sigma_S \bm V_S^T = \bm S$ be an SVD of $\bm S$, where $\bm U_S \in \mathcal{O}^{d_L-r}$, $\bm V_S \in \mathcal{O}^{d_0-r}$, and $\bm \Sigma_S \in \R^{(d_L-r)\times (d_0-r)}$. Substituting this into \eqref{eq5:thm opti G} and rearranging the terms yields
\begin{align}\label{eq6:thm opti G}
\begin{bmatrix}
\bm I_r & \bm 0 \\
\bm 0 & \bm U_S^T
\end{bmatrix}\bm U_L^T\bm Y\bm V_1\bm Q^T \begin{bmatrix}
\bm I_r & \bm 0 \\
\bm 0 & \bm V_S
\end{bmatrix} =\begin{bmatrix}
\frac{1}{\sqrt{\lambda}} \left( \tilde{\bm \Sigma}^{L} + \lambda\tilde{\bm \Sigma}^{2-L} \right) & \bm 0\\
\bm 0 & \bm \Sigma_S
\end{bmatrix}.
\end{align}
Since  $\mathrm{BlkD}\left(\bm I, \bm U_S^T\right)\bm U_L^T \in \mathcal{O}^{d_L}$ and $\bm V_1\bm Q^T\mathrm{BlkD}\left(\bm I, \bm V_S\right) \in \mathcal{O}^{d_0}$, the above left-hand side is an SVD of the right-hand diagonal matrix. Therefore, we obtain that the diagonal elements of $ \mathrm{BlkD}(( \tilde{\bm \Sigma}^{L} + \lambda\tilde{\bm \Sigma}^{2-L})/\sqrt{\lambda}, \bm \Sigma_S)$ are a permutation of those of $\bm Y$.
Thus,
there exists a permutation matrix $\bm \Pi \in \mathcal{P}^{d_Y}$ such that
\begin{align}
\notag
	 \begin{bmatrix}
\frac{1}{\sqrt{\lambda}} \left( \tilde{\bm \Sigma}^{L} + \lambda\tilde{\bm \Sigma}^{2-L} \right) & \bm 0\\
\bm 0 & \bm \Sigma_S
\end{bmatrix}  = \begin{bmatrix}
	\bm \Pi & \bm 0\\ \bm 0 & \bm I_{d_L - d_{Y}}
	\end{bmatrix} \bm Y \begin{bmatrix}
	\bm \Pi^T & \bm 0\\ \bm 0 & \bm I_{d_0 - d_{Y}}
	\end{bmatrix}.
\end{align}
Substituting this into \eqref{eq6:thm opti G} yields
\begin{align*}
\bm Y = \left(\begin{bmatrix}
	\bm \Pi^T & \bm 0\\ \bm 0 & \bm I_{d_L - d_{Y}}
	\end{bmatrix} \begin{bmatrix}
\bm I_r & \bm 0 \\
\bm 0 & \bm U_S^T
\end{bmatrix} \bm U_L^T\right) \bm Y \left( \bm V_1\bm Q^T \begin{bmatrix}
\bm I_r & \bm 0 \\
\bm 0 & \bm V_S
\end{bmatrix} \begin{bmatrix}
	\bm \Pi & \bm 0\\ \bm 0 & \bm I_{d_0 - d_{Y}}
	\end{bmatrix}\right).
\end{align*}
Since the right-hand side is an SVD of a diagonal matrix $\bm Y$ in \eqref{eq:Y} and \eqref{eq:SY}, there exist $\bm O_1 \in \mathcal{O}^{h_1},\dots,\bm O_{p_Y} \in \mathcal{O}^{h_{p_Y}},\bm O_{p_Y+1} \in \mathcal{O}^{d_0-r_Y}$ and $\widehat{\bm O}_{p_Y+1} \in \mathcal{O}^{d_L-r_Y}$ such that
\begin{align*}
& \begin{bmatrix}
	\bm \Pi^T & \bm 0\\ \bm 0 & \bm I_{d_0 - d_Y}
	\end{bmatrix} \begin{bmatrix}
\bm I_r & \bm 0 \\
\bm 0 & \bm V_S^T
\end{bmatrix}\bm Q \bm V_1^T = \mathrm{BlkD}\left( \bm O_1,\dots,\bm O_{p_Y},\bm O_{p_Y+1} \right), \\
& \begin{bmatrix}
	\bm \Pi^T & \bm 0\\ \bm 0 & \bm I_{d_L - d_{Y}}
	\end{bmatrix} \begin{bmatrix}
\bm I_r & \bm 0 \\
\bm 0 & \bm U_S^T
\end{bmatrix} \bm U_L^T = \mathrm{BlkD}\left( \bm O_1,\dots,\bm O_{p_Y},\widehat{\bm O}_{p_Y+1} \right).
\end{align*}
This implies
\begin{align}
& \bm V_1 = \mathrm{BlkD}\left( \bm O_1^T,\dots,\bm O_{p_Y}^T,\bm O_{p_Y+1}^T \right)\blk\left( \bm \Pi^T, \bm I\right) \blk\left(\bm I_r, \bm V_S^T \right) \bm Q,\label{eq9:thm opti G} \\
& \bm U_L = \mathrm{BlkD}\left( \bm O_1^T,\dots,\bm O_{p_Y}^T,\widehat{\bm O}_{p_Y+1}^T \right)\blk\left( \bm \Pi^T, \bm I\right) \blk\left(\bm I_r, \bm U_S^T \right). \label{eq8:thm opti G}
\end{align}
Substituting \eqref{eq3:bala} and \eqref{eq8:thm opti G} into \eqref{eq:svdwl} yields
\begin{align}\label{eq:defwL}
\bm W_L = \mathrm{BlkD}\left( \bm O_1^T,\dots,\bm O_{p_Y}^T,\widehat{\bm O}_{p_Y+1}^T \right)\mathrm{BlkD}\left(\bm \Pi^T, \bm I \right)\mathrm{BlkD}\left(\tilde{\bm \Sigma}, \bm 0 \right) \bm V_L^T.
\end{align}
Next, substituting \eqref{eq:Ul Vl} and \eqref{eq:u1v1} into \eqref{eq:svdwl} yields
\begin{align*}
    \bm W_{L-1}&= \bm V_{L} \mathrm{BlkD}\left(\bm Q_{L-1}^{(1)},\dots,\bm Q_{L-1}^{(p)}, \bm I \right)
    \mathrm{BlkD}\left(\tilde{\bm \Sigma}, \bm 0 \right) \bm V_{L-1}^T = \bm V_L \mathrm{BlkD}(\tilde{\bm \Sigma}, \bm 0) \bm P_{L-1}^T,
\end{align*}
where $\bm P_{L-1} := \bm V_{L-1}\mathrm{BlkD}\left( \bm{Q}_{L-1}^{(1)^T},\dots, \bm{Q}_{L-1}^{(p)^T}, \bm I \right) \in \mathcal{O}^{d_{L-2}}$. Using the same argument, we obtain
\begin{align}\label{eq:defwl}
\bm W_{l} = \bm P_{l+1}\mathrm{BlkD}(\tilde{\bm \Sigma}, \bm 0)\bm P_{l}^T,\ l = 2,\dots,L-2,
\end{align}
where $\bm P_{l} := \bm V_{l}\mathrm{BlkD}\left( \prod_{j=l}^{L-1} \bm Q_{j}^{(1)^T},\dots, \prod_{j=l}^{L-1} \bm Q_{j}^{(p)^T}, \bm I\right) \in \mathcal{O}^{d_{l-1}}$ for $l=2,\dots,L-2$. Finally, using \eqref{eq:Qexpression}  and \eqref{eq9:thm opti G}, we compute
\begin{align}\label{eq:defw1}
\bm W_1 & \overset{(\ref{eq:svdwl}, \ref{eq:u1v1})}{=} \bm P_{2}\mathrm{BlkD}(\tilde{\bm \Sigma}, \bm 0)\mathrm{BlkD}\left( \prod_{j=1}^{L-1} \bm Q_{j}^{(1)},\dots, \prod_{j=1}^{L-1} \bm Q_{j}^{(p)}, \bm I\right)\bm V_1^T \notag \\
& \overset{(\ref{eq:Qexpression}, \ref{eq9:thm opti G})}{=} \bm P_2 \mathrm{BlkD}(\tilde{\bm \Sigma}, \bm 0) \mathrm{BlkD}(\bm \Pi,\bm I) \mathrm{BlkD}\left( \bm O_1,\dots,\bm O_{p_Y},\bm O_{p_Y+1} \right).
\end{align}
Now, we define $\bm \sigma := \left(\sigma_1, \sigma_2, \dots, \sigma_{d_{\min}} \right) \in \R^{d_{\min}}$ and $\tilde{\bm \sigma} := (\bm \sigma, \bm 0) \in \R^{d_Y}$. We write $\bm \Sigma_l \in \R^{d_l\times d_{l-1}}$ as $\bm \Sigma_l =   \mathrm{BlkD}\left(\mathrm{diag}(\bm \sigma), \bm 0\right)$  for each $l \in [L].$
Next, it remains to show that $(\bm \sigma, \bm \Pi) \in \mathcal{B}$. Substituting \eqref{eq:defwL}, \eqref{eq:defwl}, and \eqref{eq:defw1} into $\nabla_{L} G(\bm W) = \bm 0$ yields
\begin{align*}
  & \nabla_L G(\bm W) \overset{\eqref{eq:grad G}}{=} 2\bm W_{L:1}\bm W_{L-1:1}^T - 2\sqrt{\lambda}\bm Y\bm W_{L-1:1}^T +2 \lambda \bm W_L\\
 & =2 \mathrm{BlkD}\left(\bm O_1^T,\dots,\bm O_{p_{Y}}^T,\widehat{\bm O}_{p_{Y}+1}^T\right)\mathrm{BlkD}\left(\bm \Pi^T (\mathrm{diag}^{2L-1}(\tilde{\bm \sigma})+\lambda \mathrm{diag}(\tilde{\bm \sigma})), \bm 0 \right) \bm V_L^T\\
& \quad  - 2\sqrt{\lambda}\bm Y \mathrm{BlkD}\left( \bm O_1^T,\dots,\bm O_{p_{Y}}^T,\bm O_{p_{Y}+1}^T \right)\mathrm{BlkD} (\bm \Pi^T\mathrm{diag}^{L-1}(\tilde{\bm \sigma}),\bm 0)\bm V_L^T=\bm 0.
\end{align*}
This, together with
the diagonal structure in \eqref{eq:SY}, implies
\begin{align*}
\mathrm{BlkD}\left(\bm \Pi^T (\mathrm{diag}^{2L-1}(\tilde{\bm \sigma})+\lambda \mathrm{diag}(\tilde{\bm \sigma})), \bm 0 \right) =
\sqrt{\lambda}\bm Y \mathrm{BlkD}\left(\bm \Pi^T\mathrm{diag}^{L-1}(\tilde{\bm \sigma}), \bm 0 \right).
\end{align*}
This implies $\bm{\Pi}^T \tilde{\bm{\sigma}} \in \mathcal{A}$. Hence, there exists $\bm{a} \in \mathcal{A}$ with $\bm{\Pi}\bm{a} = \tilde{\bm{\sigma}}$, giving $(\bm{\sigma}, \bm{\Pi}) \in \mathcal{B}$.
\end{proof}
Equipped with \Cref{lem:equi FG} and \Cref{prop:opti G}, we are ready to prove \Cref{thm:crit}, which characterizes the critical point set $\mathcal{W}_F$ of Problem \eqref{eq:F}.
\begin{proof}[Proof of \Cref{thm:crit}]
By \Cref{lem:equi FG}(i), every critical point in $\mathcal{W}_F$ can be characterized using the explicit form of $\mathcal{W}_G$, which is given in \Cref{prop:opti G}.
Let $(\bm{W}_1,\dots,\bm{W}_L) \in \mathcal{W}_G$ be any critical point.
Combining \Cref{lem:equi FG}(i) with the SVD $\bm{Y} = \bm{U}_Y \bm{\Sigma}_Y \bm{V}_Y^T$ yields
$$(\bm W_1\bm V_Y^T/\sqrt{\lambda_1},\bm W_2/\sqrt{\lambda_2},\dots,\bm W_{L-1}/\sqrt{\lambda_{L-1}},\bm U_Y\bm W_L/\sqrt{\lambda_L}) \in \mathcal{W}_F.$$
This, together with the definitions of $\mathcal{A}$ and $\mathcal{B}$, yields \eqref{eq:crit sol} and \eqref{eq:thmsimga pi}.
\end{proof}
We now show that all $\bm{W} \in \mathcal{W}_{(\bm{\sigma}, \bm{\Pi})}$ share the same landscape for any $(\bm{\sigma}, \bm{\Pi}) \in \mathcal{B}$.

\begin{lemma}\label{prop:equivalent}
Suppose $L \ge 3$. Let $(\bm{\sigma}, \bm{\Pi}) \in \mathcal{B}$ be arbitrary. Then for any pair $\bm{W},\, \hat{\bm{W}} \in \mathcal{W}_{(\bm{\sigma}, \bm{\Pi})}$,
$\bm{W}$ is a local minimizer, local maximizer, global minimizer, strict saddle point, or non-strict saddle point of Problem~\eqref{eq:G} if and only if $\hat{\bm{W}}$ is of the same type.
\end{lemma}
\begin{proof}
According to \Cref{prop:opti G}, there exist $\bm Q_l\in\mathcal O^{d_{l-1}}$ for $l=2,\ldots,L$, $\bm{O}_i \in \mathcal{O}^{h_i}$ for each $i \in [p_{Y}]$, $\bm{O}_{p_{Y}+1} \in \mathcal{O}^{d_0 - r_{Y}}$, and $\widehat{\bm{O}}_{p_{Y}+1} \in \mathcal{O}^{d_L - r_{Y}}$ such that $\bm W = (\bm W_1,\dots,\bm W_L)$ satisfies
    \begin{align*}
    \left\{
    \begin{array}{l}
        \bm{\Sigma}_l = \mathrm{BlkD}\left(\mathrm{diag}(\bm{\sigma}), \bm{0}\right) \in \mathbb{R}^{d_l\times d_{l-1}},\ \forall l \in [L],\\
        \bm{W}_1 = \bm{Q}_2\bm{\Sigma}_1 \mathrm{BlkD}\left(\bm{\Pi}, \bm{I}\right)\mathrm{BlkD}\left(\bm{O}_1,\dots,\bm{O}_{p_{Y}},\bm{O}_{p_{Y}+1}\right), \\
        \bm{W}_l = \bm{Q}_{l+1} \bm{\Sigma}_l \bm{Q}_l^T,\ l=2,\dots,L-1,\ \bm{Q}_l \in \mathcal{O}^{d_{l-1}},\ l=2,\dots,L,  \\
        \bm{W}_L = \mathrm{BlkD}\left(\bm{O}_1^T,\dots,\bm{O}_{p_{Y}}^T,\widehat{\bm{O}}_{p_{Y}+1}^T\right)\mathrm{BlkD}\left(\bm{\Pi}^T, \bm{I}\right)\bm{\Sigma}_L \bm{Q}_L^T.
    \end{array}
    \right.
    \end{align*}
Now, let $\bm W^\prime = (\bm W_1^\prime,\dots,\bm W_L^\prime)$ be such that
\begin{align*}
    \bm{W}_1^\prime  = \bm{\Sigma}_1 \mathrm{BlkD}(\bm{\Pi}, \bm{I}),\ \bm{W}_l^\prime = \bm{\Sigma}_l, \ l=2,\dots,L-1,\ \bm{W}_L^\prime = \mathrm{BlkD}(\bm{\Pi}^T, \bm{I}) \bm{\Sigma}_L,
\end{align*}
For any $\bm Z = (\bm Z_1,\dots,\bm Z_L) \in \R^{d_1\times d_0} \times \dots \times \R^{d_{L}\times d_{L-1}}$, we construct a linear bijection
    \begin{align}
    \psi(\bm{Z}_1,\ldots,\bm{Z}_L) &= \left(\bm{Q}_2\bm{Z}_1\mathrm{BlkD}\left(\bm{O}_1,\dots,\bm{O}_{p_{Y}},\bm{O}_{p_{Y}+1}\right), \bm{Q}_3\bm{Z}_2\bm{Q}_2^T,\ldots, \right. \notag \\
    &\quad \left. \bm{Q}_L \bm{Z}_{L-1}\bm{Q}_{L-1}^T, \mathrm{BlkD}\left(\bm{O}_1^T,\dots,\bm{O}_{p_{Y}}^T,\widehat{\bm{O}}_{p_{Y}+1}^T\right)\bm{Z}_L\bm{Q}_L^T\right).\notag
    \end{align}
    Then, one can verify that $ G(\bm{Z}) = G(\psi(\bm{Z})) $ holds for all $\bm{Z}$.
    Using this and \Cref{lem:boumalcite} and noting that $\bm W \in \mathcal{W}_{(\bm{\sigma}, \bm{\Pi})}$ is arbitrary, we obtain the desired result.
\end{proof}
Next, we present a lemma that characterizes the properties of the singular values belonging to the sets defined in \eqref{eq:defsetgamma}. Since the proof of the lemma only involves standard techniques from elementary calculus, we defer it to \Cref{subsec:proofderprop}.

\begin{lemma}\label{lem:derprop}
{
Let $f(x;y)$ be defined in \eqref{eq:deffty}, $L\ge 3$, and
\begin{align}
& x_* := \left(\frac{L-2}{L}\right)^{\frac{1}{2L-2}}\lambda^{\frac{1}{2L-2}},\label{eq:x*}\\
& y_* := \left( \left( \frac{L-2}{L} \right)^{\!\frac{L}{2L-2}}
+ \left( \frac{L}{L-2} \right)^{\!\frac{L-2}{2L-2}} \right)
\lambda^{\frac{1}{2L-2}}. \label{eq:defnewbary}
\end{align}
The following statements hold:  \\
(i) If $y > y_*$, $f(x; y) = 0$ has two distinct positive roots; if $y = y_*$, it has a unique positive root; and otherwise it has no positive root. \\
(ii) If $y > y_*$, the larger positive root $\overline{x}(y)$ of $f(x; y) = 0$ is strictly increasing in $y$, while its smaller positive root $\underline{x}(y)$ is strictly decreasing in $y$.\\
(iii) If there exists $i\in [d_Y]$ such that $y_i > y_*$, then $f(x; y_i) = 0$ has two distinct positive roots $\overline{x}(y_i) \in \mathcal{S}_1$ satisfying ${\partial_x f(\overline{x}(y_i); y_i)} > 0$ and $\underline{x}(y_i) \in \mathcal{S}_2$ satisfying ${\partial_x f(\underline{x}(y_i); y_i)} < 0$. \\
(iv) If $\mathcal{S}_3 \neq \emptyset$, then $\mathcal{S}_3 = \{x_*\}$.  If there exists $i\in [d_Y]$ such that $y_i = y_*$, then $f(x; y_i) = 0$ has a unique positive root $\hat{x}(y_i) =x_*$ satisfying $\tfrac{\partial f(x_*; y_i)}{\partial x} = 0 $ and $\tfrac{\partial^2 f(x_*; y_i)}{\partial x^2} > 0$. \\
(v) For any $z_1 \in \mathcal{S}_1$ and $z_2 \in \mathcal{S}_2$, it holds that $z_1 > x_* > z_2$. \\
}
\end{lemma}

\subsection{Analysis of the Loss Landscape}\label{subsec:land}

In this subsection, we analyze the loss landscape of Problem \eqref{eq:G}. Notably, our proof follows the steps outlined in \Cref{fig2:env}.
Let \( (\bm{\sigma}, \bm{\Pi}) \in \mathcal{B} \) and \( \bm{W} = (\bm{W}_1, \ldots, \bm{W}_L) \in \mathcal{W}_{(\bm{\sigma}, \bm{\Pi})} \) be arbitrary. By \Cref{prop:equivalent}, all points in $\mathcal{W}_{(\bm{\sigma}, \bm{\Pi})}$ are of the same critical point type. This, together with \eqref{set:sigmapi}, implies that it suffices to study  $\bmw = (\bm{W}_1, \ldots, \bm{W}_L) \in \mathcal{W}_{(\bm{\sigma}, \bm{\Pi})} $ of the following form:
    \begin{align}\label{eq:simplification}
    \begin{cases}
        \bm \Sigma_l = \mathrm{BlkD}\left(\mathrm{diag}(\bm \sigma), \bm 0\right) \in \R^{d_l\times d_{l-1}},\ \forall l \in [L], \\
        \bm W_1 = \bm \Sigma_1 \blk(\bm \Pi, \bm I), \bm W_l = \bm \Sigma_l , l=2,\dots,L-1, \bm W_L =  \blk(\bm \Pi^T, \bm I)\bm \Sigma_L.
    \end{cases}
    \end{align}
For convenience, define $\widehat{\bm{Y}} := \blk(\bm{\Pi}, \bm{I}) \bm{Y}\blk(\bm{\Pi}^T\!,\bm{I})$. Using \eqref{eq:Y}, we have
\begin{align}\label{eq:Yhat}
    \widehat{\bm Y} = \blk\left( \mathrm{diag}(y_{\pi(1)},\dots,y_{\pi(d_Y)}), \bm 0_{(d_L-d_Y)\times (d_0-d_Y)} \right).
\end{align}
When $\bm W$ takes the form of \eqref{eq:simplification}, we obtain
\begin{align}\label{eq:calculation}
     \|\bm W_L\cdots\bm W_1 - \sqrt{\lambda} \bm Y \|_F^2 = \|\bm \Sigma_L\cdots\bm \Sigma_1 - \sqrt{\lambda}\widehat{\bm Y}\|_F^2.
\end{align}
With the above setup, we first show that Problem \eqref{eq:G} admits no local maximizers. Using the function  $f(x;y)$ defined in \eqref{eq:deffty}, we rewrite the set \(\mathcal{B}\) defined in \Cref{prop:opti G} as
\begin{align} \label{set:newb}
\mathcal{B} = \left\{ (\bm \sigma, \bm \Pi) \in \mathbb{R}^{d_{\min}} \times \mathcal{P}^{d_Y} : f(\sigma_i; y_{\pi(i)}) = 0,\ \forall i,\ \sigma_1 \ge \cdots \ge \sigma_{d_{\min}} \ge 0 \right\},
\end{align}
where \( \pi : [d_Y] \to [d_Y] \) is the permutation corresponding to \( \bm \Pi \).
\begin{lemma}\label{lem:nonlocalmax}
    Suppose that $L \ge 3$ and $\bm Y$ is defined in \eqref{eq:Y}. Let \( (\bm{\sigma}, \bm{\Pi}) \in \mathcal{B} \)  and \( \bm{W} = (\bm{W}_1, \ldots, \bm{W}_L) \in \mathcal{W}_{(\bm{\sigma}, \bm{\Pi})} \) be arbitrary. Then $\bm W$ is not a local maximizer.
\end{lemma}
\begin{proof}
Consider that \( \bm{W} \) takes the form in \eqref{eq:simplification}.
Since \(\max_{0\le j\le L} d_j\ge2\), we may, without loss of notation,
consider the case \(d_{L-1}\ge2\); the other cases are handled by applying
the same construction to the corresponding layer. We construct a direction
\( \bm{\Delta} = (\bm \Delta_1,\dots,\bm \Delta_L) \in
\R^{d_1\times d_0} \times \dots \times \R^{d_L\times d_{L-1}}\)
with \(\bm{\Delta}_{l} = \bm 0\) for each \(l \in [L-1]\) and
\(\bm{\Delta}_{L} = \blk(\bm{\Pi}^T, \bm{I}) \bm{\Theta}_L\), where
\(\bm{\Theta}_L \in \mathbb{R}^{d_L \times d_{L-1}}\) satisfies
\(\bm{\Theta}_L(1,2)=1\) and all its other entries are zero.
    According to \eqref{eq:G} and \eqref{eq:simplification}, we compute
\begin{align*}
    G(\bm W+t\bm \Delta) & = \|(\bm W_L+t\bm \Delta_L)\bm W_{(L-1):1} - \sqrt{\lambda}\bm Y\|_F^2 \\
    &\quad + \lambda \left( \sum_{l=1}^{L-1}\|\bm W_l\|_F^2 + \|\bm W_L+t\bm \Delta_L\|_F^2\right)  = G(\bm W) +\sigma_2^{2L-2}t^2+\lambda t^2.
\end{align*}
For any $t \neq 0$, it holds that $G(\bm W+t\bm \Delta) > G(\bm W) $. Thus, $\bm W$ is not a local maximizer.
\end{proof}

According to this lemma and \Cref{def:crit}(iv), we know that a critical point is either a saddle point or a local minimizer. Now, we show that a critical point $\bm W$ is a strict saddle point if it has a singular value $\sigma_i \in \mathcal{S}_2$ (see \eqref{eq:defsetgamma}).

\begin{proposition}\label{prop:gamma2}
Suppose that $L \ge 3$ and $\bm Y$ is defined in \eqref{eq:Y}. Let \( (\bm{\sigma}, \bm{\Pi}) \in \mathcal{B} \) be arbitrary and \( \bm{W} = (\bm{W}_1, \ldots, \bm{W}_L) \in \mathcal{W}_{(\bm{\sigma}, \bm{\Pi})} \) be any critical point. Suppose in addition that there exists some \(i \in [r_\sigma] \) such that \( \sigma_i \in \mathcal{S}_2 \). Then, \( \bm{W} \) is a strict saddle point, where $r_{\sigma} := \|\bm \sigma\|_0$.
\end{proposition}
\begin{proof}
    Consider that \( \bm{W} \) takes the form in \eqref{eq:simplification}. 
     We construct \( \bm{\Delta}  \) as $\bm{\Delta}_1 = \bm{E}_1\blk(\bm{\Pi}, \bm{I})$, $\bm{\Delta}_{l} = \bm{E}_l$, $l = 2, \dots, L-1$, and $\bm{\Delta}_{L} = \blk(\bm{\Pi}^T, \bm{I}) \bm{E}_L$. Here, the $(i,i)$-th entry of $\bm{E}_l \in \mathbb{R}^{d_l \times d_{l-1}}$  is $-1$ and all other entries are $0$ for each $l \in [L]$.
According to \eqref{eq:Y}, \eqref{eq:simplification}, and \eqref{eq:calculation}, we compute
    \begin{align*}
        G(\bmw) &= \left\| \prod_{l=L}^1\bm \Sigma_l - \sqrt{\lambda}\widehat{\bm Y} \right\|_F^2 + \lambda\sum_{l=1}^L \|\bm \Sigma_l\|_F^2 \\
        &=\sum_{j=1}^{d_{\min}}\left(\sigma_j^L-\sqrt{\lambda}y_{\pi(j
        )}\right)^2+\sum_{j=d_{\min}+1}^{d_Y} \lambda y_{\pi(j)}^2 +\lambda L\sum_{j=1}^{d_{\min}}\sigma_j^2, \\
        G(\bmw+t \bm \Delta)
&=\sum_{j=1,j\neq i}^{d_{\min}}\left((\sigma_j^L-\sqrt{\lambda}y_{\pi(j)})^2+ \lambda L\sigma_j^2\right) +  \sum_{j=d_{\min}+1}^{d_Y} \lambda y_{\pi(j)}^2\\
&\qquad + \left((\sigma_i-t)^L-\sqrt{\lambda}y_{\pi(i)}\right)^2 + \lambda L(\sigma_i-t)^2,
    \end{align*}
    where $\pi:[d_Y] \to [d_Y]$ is the permutation corresponding to $\bm \Pi$.
Based on the above two equations, we compute
\begin{align*}
& \qquad G(\bmw+t\bm\Delta) - G(\bmw) \\
& = \left((\sigma_i - t)^L - \sqrt{\lambda} y_{\pi(i)}\right)^2 + \lambda L (\sigma_i - t)^2 -\left(\sigma_i^L - \sqrt{\lambda} y_{\pi(i)}\right)^2 - \lambda L \sigma_i^2\\
& =
-2L \left(\sigma_i^{2L-1}-\sqrt{\lambda}y_{\pi(i)}\sigma_i^{L-1}+\lambda\sigma_i\right)t \\
&\quad\ + L\left((2L-1)\sigma_i^{2L-2}-\sqrt{\lambda}(L-1)y_{\pi(i)}\sigma_i^{L-2}+\lambda \right)t^2
+ O(t^3)\\
& = -2L f(\sigma_i;y_{\pi(i)}) t +L\partial_x f(\sigma_i;y_{\pi(i)}) t^2 +O(t^3)
= L\partial_x f(\sigma_i;y_{\pi(i)}) t^2 + O(t^3),
\end{align*}
where the last equality uses $f(\sigma_i; y_{\pi(i)}) =0$ due to  $(\bm\sigma, \bm \Pi)\in \mathcal{B}$ and \eqref{set:newb}. This, together with $\nabla G(\bm W) = \bm 0$ and \eqref{eq:defhessian}, yields $\nabla^2 G(\bm W)[\bm \Delta,\bm \Delta] = 2L\partial_x f(\sigma_i; y_{\pi(i)}) < 0$, where the inequality uses $\sigma_i \in \mathcal{S}_2$ and $\partial_x f(\sigma_i; y_{\pi(i)}) < 0$ due to { \Cref{lem:derprop}(iii)}.
\end{proof}
Following the proof roadmap in \Cref{fig2:env}, we proceed to identify other conditions under which $\bm{W} \in \mathcal{W}_{(\bm{\sigma}, \bm{\Pi})}$ is a strict saddle point.
\begin{proposition}\label{prop:misalign}
Suppose $L \ge 3$ and $\bm Y$ is defined in \eqref{eq:Y}. Let $(\bm{\sigma}, \bm{\Pi}) \in \mathcal{B}$ be arbitrary and $\bm{W} = (\bm{W}_1, \ldots, \bm{W}_L) \in \mathcal{W}_{(\bm{\sigma}, \bm{\Pi})}$ be any critical point, and define $r_\sigma := \|\bm{\sigma}\|_0$. If $\sigma_i \in \mathcal{S}_1\cup\mathcal{S}_3$ for each $i \in [r_\sigma]$ and
$(y_{\pi(1)},\ldots,y_{\pi(r_\sigma)}) \neq (y_1,\ldots,y_{r_\sigma})$, $\bm{W}$ is a strict saddle point.
\end{proposition}
\begin{proof}
Consider that \( \bm{W} \) takes the form in \eqref{eq:simplification}.
 Note that \( f(\sigma_i; y_{\pi(i)}) = 0 \) for each $i \in [r_{\sigma}]$ according to \eqref{set:newb}. This, together with \( \sigma_i \in \mathcal{S}_1 \cup \mathcal{S}_3 \) for all \( i \in [r_\sigma] \), $\sigma_1\ge \dots \ge \sigma_{r_{\sigma}}$, and {  \Cref{lem:derprop}(ii) and (v)}, yields \( y_{\pi(1)} \geq \cdots \geq y_{\pi(r_\sigma)} \).
 We claim that there exist $i, j$ such that $i \leq r_\sigma < j$ and $y_{\pi(i)} < y_{\pi(j)}$. Now, we prove the claim by contradiction. Suppose that the claim does not hold. This implies $y_{\pi(i)}\geq y_{\pi(j)}$ for all $i \in [r_\sigma]$ and all $j \in \{r_\sigma + 1,\dots,d_Y\}$.  Therefore, $y_{\pi(1)}, \ldots, y_{\pi(r_\sigma)}$ must be the $r_\sigma$ largest entries of $\bm y$. This, together with \( y_{\pi(1)} \geq \cdots \geq y_{\pi(r_\sigma)} \), implies $(y_{\pi(1)}, \ldots, y_{\pi(r_\sigma)}) = (y_1, \ldots, y_{r_\sigma})$, which contradicts $(y_{\pi(1)}, \ldots, y_{\pi(r_\sigma)})\neq(y_1, \ldots, y_{r_\sigma})$.
Next, we construct a descent direction $\bm{\Delta} = (\bm{\Delta}_1, \ldots, \bm{\Delta}_L) \in \R^{d_1 \times d_0} \times \dots \times \R^{d_L \times d_{L-1}}$ as follows:
\[
    \bm{\Delta}_1 = \bm{E}_1\, \mathrm{BlkD}(\bm{\Pi}, \bm{I}), \quad
    \bm{\Delta}_l = \bm{0},\; l = 2, \ldots, L-1,\quad
    \bm{\Delta}_L = \mathrm{BlkD}(\bm{\Pi}^T, \bm{I})\, \bm{E}_L,
\]
where $\bm{E}_1 \in \R^{d_1 \times d_0}$ has a 1 at entry $(i,j)$ and 0 elsewhere, and $\bm{E}_L \in \R^{d_L \times d_{L-1}}$ has a 1 at entry $(j,i)$ and 0 elsewhere.
    Using this, \eqref{eq:simplification}, and \eqref{eq:calculation}, we compute
    \begin{align*}
        & G(\bm{W} + t\bm{\Delta})
         \!=\! \left\|(\bm{\Sigma}_L + t\bm{E}_L) \bm \Sigma_{(L-1):2}(\bm{\Sigma}_1 + t\bm{E}_1) - \sqrt{\lambda}\widehat{\bm Y} \right\|_F^2 + \lambda\sum_{l=1}^L\|\bmw_l+t\bm\Delta_l\|_F^2\\
        &= G(\bmw) + \left\| t\bm E_L\bm \Sigma_{(L-1):1}+t\bm \Sigma_{L:2}\bm E_1 + t^2 \bm E_L \bm \Sigma_{(L-1):2} \bm E_1\right\|_F^2 +2\lambda t^2\\
        &\quad + 2\left\langle \bm\Sigma_{L:1}-\sqrt{\lambda}\widehat{\bm Y},  t\bm E_L\bm \Sigma_{(L-1):1}+t\bm \Sigma_{L:2}\bm E_1 + t^2 \bm E_L \bm \Sigma_{(L-1):2}\bm E_1
       \right \rangle\\
        &= G(\bmw)+2\left(\sigma_{i}^{2L-2} + \lambda -  \sqrt{\lambda}y_{\pi(j)}\sigma_{i}^{L-2}\right)t^2 +2\sigma_j^{L}\sigma_i^{L-2}t^2+ \sigma_i^{2L-4}t^4\\
        &= G(\bmw) + 2\sqrt{\lambda}\sigma_{i}^{L-2}(y_{\pi(i)} - y_{\pi(j)})t^2 + \sigma_i^{2L-4}t^4,
    \end{align*}
where the final equality follows from $f(\sigma_i,y_{\pi(i)})=0$ and $\sigma_j =0$ for $j>r_\sigma$.
This, together with \eqref{eq:defhessian} and $\nabla G(\bm{W}) = \bm{0}$, yields
$
    \nabla^2 G(\bm{W})[\bm{\Delta}, \bm{\Delta}] = 4\sqrt{\lambda}\,\sigma_{i}^{L-2}(y_{\pi(i)} - y_{\pi(j)}) < 0
$
due to $y_{\pi(i)} < y_{\pi(j)}$. Thus, $\bm{W}$ is a strict saddle point.
\end{proof}
Now, we give a sufficient condition for $\bm{W}\!\in\! \mathcal{W}_{(\bm\sigma,\bm \Pi)}$ to be a local minimizer.
To proceed, we define the following auxiliary function:
\begin{align}\label{eq:defg}
    g(x; y) := (x^L - \sqrt{\lambda} y)^2 + \lambda L x^2.
\end{align}
\vspace{-0.25in}
\begin{proposition}\label{prop:l3local}
Suppose that $L \ge 3$ and $\bm Y$ is defined in \eqref{eq:Y}. Let $(\bm{\sigma}, \bm{\Pi}) \in \mathcal{B}$ be arbitrary and $\bm{W} = (\bm{W}_1, \ldots, \bm{W}_L) \in \mathcal{W}_{(\bm{\sigma}, \bm{\Pi})}$ be any critical point, and define $r_{\sigma} := \|\bm \sigma\|_0$. If $\sigma_i \in \mathcal{S}_1$ for all $i \in [r_\sigma]$ and $(y_{\pi(1)},\ldots,y_{\pi(r_\sigma)}) = (y_1,\ldots,y_{r_\sigma})$, then $\bm{W}$ is a local minimizer.
\end{proposition}
\begin{proof}
Consider that \( \bm{W} \) takes the form in \eqref{eq:simplification}.
Note that $\blk(\bm{\Pi}, \bm{I}_{d_0 - d_Y})$ and $\blk(\bm{\Pi}^T, \bm{I}_{d_L - d_Y})$ are invertible.
Thus, for any point $\widehat{\bm{W}} = (\widehat{\bm{W}}_1, \dots, \widehat{\bm{W}}_L)$ in a neighborhood of $\bm{W}$, there exists $(\bm{\Delta}_1, \ldots, \bm{\Delta}_L) \in \R^{d_1 \times d_0} \times \dots \times \R^{d_L \times d_{L-1}}$ such that
\begin{align*}
& \widehat{\bmw}_1 = (\bm \Sigma_1+\bm\Delta_1) \blk(\bm \Pi, \bm I),\ \widehat{\bm W}_L = \blk(\bm \Pi^T, \bm I)(\bm \Sigma_L+\bm\Delta_L), \\
& \widehat{\bm W}_l = \bm\Sigma_l+\bm\Delta_l,\ l=2,\dots,L-1.
\end{align*}
    Using \eqref{eq:simplification}, \eqref{eq:Yhat}, and \eqref{eq:calculation}, we compute
    \begin{align}\label{eq1:prop l3local}
        G(\bm W) &
        =\sum_{i=1}^{r_{\sigma}}\left((\sigma_i^L -\sqrt{\lambda} y_{\pi(i)})^2 +\lambda L\sigma_i^2\right) + \lambda \sum_{i=r_{\sigma}+1}^{d_{Y}}y_{\pi(i)}^2 \notag \\
        & = \sum_{i=1}^{r_{\sigma}}\left((\sigma_i^L -\sqrt{\lambda} y_i)^2 +\lambda L\sigma_i^2\right) + \lambda \sum_{i=r_{\sigma}+1}^{d_{Y}}y_i^2
        = \sum_{i=1}^{r_{\sigma}}g(\sigma_i;y_i
        )+\sum_{i= r_{\sigma}+1}^{d_Y} g(0;y_i),
      \end{align}
    where the second equality uses $(y_{\pi(1)},\ldots,y_{\pi(r_\sigma)}) = (y_1,\ldots,y_{r_\sigma})$ and the last equality is due to \eqref{eq:defg}. For ease of exposition, we define $\widetilde{\bm W} := (\bm \Sigma_L + \bm \Delta_L)\cdots(\bm\Sigma_1+\bm\Delta_1)$. Applying \Cref{lem:mirsky} to $\prod_{l=L}^1(\bm \Sigma_l + \bm \Delta_l) -\sqrt{\lambda} \widehat{\bm Y}$ yields
    \begin{align}\label{eq5:prop l3local}
        \left\|\prod_{l=L}^1(\bm \Sigma_l + \bm \Delta_l) -\sqrt{\lambda} \widehat{\bm Y}\right\|_F^2 &\ge \sum_{i=1}^{d_Y} \left(\sigma_i(\widetilde{\bmw})-\sqrt{\lambda}y_i\right)^2.
    \end{align}
    Applying \Cref{lem:schattenp} to $\sum_{l=1}^{L} \left\|\bm\Sigma_l + \bm\Delta_l \right\|_F^2$ with $p \!=\!2/L$ and $p_l \!=\! 2$ for $l \in [L]$ yields
    \begin{align}\label{eq6:prop l3local}
        \sum_{l=1}^{L} \left\|\bm\Sigma_l + \bm\Delta_l \right\|_F^2 \ge L \sum_{i=1}^{d_Y}\sigma_i^{\frac{2}{L}}(\widetilde{\bmw}),
    \end{align}
    where the Frobenius norm is the Schatten-2 norm.
     Now, we compute
    \begin{align}\label{eq2:prop l3local}
        G(\widehat{\bmw})
        &= \left\| \prod_{l=L}^1(\bm \Sigma_l + \bm \Delta_l) -\sqrt{\lambda} \widehat{\bm Y} \right\|_F^2 + \lambda \sum_{l=1}^{L} \left\|\bm\Sigma_l + \bm\Delta_l \right\|_F^2 \notag\\
        & \overset{(\ref{eq5:prop l3local}, \ref{eq6:prop l3local})}{\ge} \sum_{i=1}^{d_Y} \left(\left(\sigma_i(\widetilde{\bmw})-\sqrt{\lambda}y_i\right)^2  + \lambda L \sigma_i^{\frac{2}{L}}(\widetilde{\bmw})\right) = \sum_{i=1}^{d_Y}g\left(\sigma_i(\widetilde{\bmw})^{\frac{1}{L}}; y_i\right).
    \end{align}
    We note that
    $ \partial_x g(\sigma_i; y_i)= 2L f(\sigma_i;y_i) = 0 $ and $\sigma_i \in \mathcal{S}_1$ for all $i\in [r_\sigma]$.
    This, together with \Cref{coro:completelocal}, yields that $ \sigma_i $ is a local minimum of $g(x; y_i)$ for all $ i \in [r_\sigma]$ and $ 0$ is a local minimum of $ g(x; y_i)$ for any $i=r_\sigma + 1,\dots, d_Y$.
    This implies for each $i \in [d_{Y}]$, there exists a $\delta_i > 0$ such that
    \begin{align}\label{eq3:prop l3local}
    \begin{cases}
            &g(\sigma_i;y_i) \le g(x;y_i),\ \forall x\ \text{satisfies}\ |x - \sigma_i| \le \delta_i,\ \forall i \in [r_{\sigma}],\\
        &g(0;y_i)\le g(x;y_i),\ \forall x\  \text{satisfies}\ |x|\le \delta_i,\ \forall i \in \{r_{\sigma}+1,\dots,d_Y\}.
    \end{cases}
    \end{align}
    Let $\bm \Delta$ be such that for each $l \in [L]$,
    \begin{equation}
    \label{eq:delta_condition}
    \|\bm \Delta_l\| \le\min \left\{ \frac{\sigma_{r_\sigma}^{L-1}\min\{\delta_j: j \in [r_\sigma]\}}{L (\sigma_{1}+1)^{L-1}},\frac{\min\{\delta_j^{L}:j = r_{\sigma}+1,\dots,d_{Y}\}}{L (\sigma_{1}+1)^{L-1}}, 1\right\},
    \end{equation}
    Using Weyl's inequality, we have
    \begin{align}
    \label{eq:sigma_inequality}
        & |\sigma_i(\widetilde{\bm W})-\sigma_i^L| \le \|\widetilde{\bm W} -\bm W_L \cdots\bmw_1\| \le\|\bm \Delta_L(\bm \bmw_{L-1}+\bm \Delta_{L-1})\cdots(\bm W_1+\bm \Delta_1)\|\notag\\
        &\quad + \|\bm W_L\bm \Delta_{L-1}(\bm W_{L-2}+\bm \Delta_{L-2})\cdots(\bmw_1+\bm \Delta_1)\| +\cdots+\|\bm W_L\cdots \bm W_{2} \bm \Delta_1\|\notag\\
        & \le \sum_{l=1}^L \left(\prod_{i=1,i\neq l
        }^L(\|\bmw_i\|+\|\bm\Delta_i\|)\right)\|\bm \Delta_l\|  \overset{\eqref{eq:delta_condition}}{\le} (\sigma_{1}+1)^{L-1}\sum_{l=1}^L \|\bm \Delta_l\|,
    \end{align}
    where the second inequality uses the triangular inequality.
    Furthermore, we note that
    \begin{align*}
        |\sigma_i(\widetilde{\bm W})-\sigma_i^L|=\left|\sigma_i^{\frac{1}{L}}(\widetilde{\bmw}) -\sigma_i\right| \left(\sigma_i^{\frac{L-1}{L}}(\widetilde{\bmw})+\sigma_i^{\frac{L-2}{L}}(\widetilde{\bmw})\sigma_i+\cdots+\sigma_i^{L-1}\right),\ \forall i\in [r_\sigma].
    \end{align*}
    Using this and $\sigma_i=0$ for all $i \in \{r_{\sigma}+1,\dots,d_Y\}$, we have
    \begin{align}
        |\sigma_i^{\frac{1}{L}}(\widetilde{\bmw}) -\sigma_i| \le \frac{|\sigma_i(\widetilde{\bm W})-\sigma_i^L|}{\sigma^{L-1}_{i}}\overset{\eqref{eq:sigma_inequality}}{\le}\left(\frac{(\sigma_{1}+1)^{L-1}}{\sigma_{r_\sigma}^{L-1}}\right) \sum_{l=1}^L \|\bm \Delta_l\|\overset{\eqref{eq:delta_condition}}{\le} \delta_i,\ \forall i \in [r_\sigma],\notag \\
        \sigma_i^{\frac{1}{L}}(\widetilde{\bmw}) \overset{\eqref{eq:sigma_inequality}}{\le} { \left((\sigma_{1}+1)^{L-1}\sum_{l=1}^L \|\bm \Delta_l\|\right)^{\frac{1}{L}}}\overset{\eqref{eq:delta_condition}}{\le} \delta_i,\ \forall i \in \{r_{
        \sigma
        }+1,\dots,d_Y\}. \label{eq4:prop l3local}
    \end{align}
     Then, we have
    \begin{align*}
       G(\widehat{\bmw}) &\overset{\eqref{eq2:prop l3local}}{\ge} \sum_{i=1}^{d_{Y}} g( \sigma_i^{\frac{1}{L}}(\widetilde{\bmw}); y_i )  \overset{(\ref{eq3:prop l3local},\ref{eq4:prop l3local})}{\ge}\sum_{i=1}^{d_Y} g( \sigma_i; y_i ) \\
       &=\sum_{i=1}^{r_{\sigma}} g( \sigma_i; y_{\pi(i)})+\sum_{i=r_\sigma+1}^{d_Y} g(0;y_{i})\overset{\eqref{eq1:prop l3local}}{=} G(\bm W),
    \end{align*}
    Since $\bm\Delta$ can be any small perturbation, we obtain that $\bm W$ is a local minimizer.
\end{proof}

For $L \geq 3$, we show that there exists $i \in [d_Y]$ such that $f(\sigma_i; y_{\pi(i)}) = 0$ has a root $\sigma_i \in \mathcal{S}_3$ if and only if $\lambda$ attains a special value.
\begin{lemma}\label{lem:illcase}
Suppose that \( L \geq 3 \) and $\bm Y$ is defined in \eqref{eq:Y}. There exists \( (\bm{\sigma}, \bm{\Pi}) \in \mathcal{B} \) and $i\in [r_\sigma]$ such that \( \sigma_i \in \mathcal{S}_3 \) if and only if it holds that
\begin{align}\label{eq:illcase1}
\exists j \in [d_Y],\ \lambda = y_j^{2(L-1)}\left( \left( \frac{L-2}{L} \right)^{\frac{L}{2(L-1)}} + \left( \frac{L}{L-2} \right)^{\frac{(L-2)}{2(L-1)}} \right)^{-2(L-1)}.
\end{align}
\end{lemma}
\begin{proof}
Suppose that there exists \( (\bm{\sigma}, \bm{\Pi}) \in \mathcal{B} \) and $i\in [r_\sigma]$ such that \( \sigma_i \in \mathcal{S}_3 \). According to \eqref{set:newb} and { \Cref{lem:derprop}}(iv), we have
\begin{subequations}\label{eq:f}
\begin{align}
    f(\sigma_i; y_{\pi(i)}) &= \sigma_i^{2L-1} - \sqrt{\lambda}y_{\pi(i)} \sigma^{L-1}_i+\lambda \sigma_i=0,\label{eq:f=0}\\
    \partial_x f(\sigma_i;y_{\pi(i)}) &=(2L -1)\sigma_i^{2L-2}-\sqrt{\lambda}(L-1)y_{\pi(i)}\sigma_i^{L-2}+\lambda=0.\label{eq:fprime=0}
\end{align}
\end{subequations}
   Applying \Cref{lem:2eqs} to solve the above equations yields
    \[
        y_{\pi(i)}= \left( \left( \frac{L-2}{L} \right)^{\frac{L}{2(L-1)}} + \left( \frac{L}{L-2} \right)^{\frac{(L-2)}{2(L-1)}} \right)  \lambda^{\frac{1}{2(L-1)}}.
    \]
   This implies that \eqref{eq:illcase1} holds.
    Conversely, suppose that \eqref{eq:illcase1} holds. It is obvious that $x^* = \left(\tfrac{\lambda(L-2)}{L}\right)^{{1}/{(2L-2)}} $ satisfies the following system
    \begin{align*}
    \begin{cases}
        & f(x; y_j) = x^{2L-1} - \sqrt{\lambda} y_j x^{L-1} + \lambda x = 0,\\
        & \partial_x f(x; y_j) = (2L-1) x^{2L-2} - \sqrt{\lambda} (L-1) y_j x^{L-2} + \lambda = 0,
    \end{cases}
    \end{align*}
    which implies that $x^*\in \mathcal{S}_3$.
We choose $\bm{\sigma} = (x^*, 0, \ldots, 0)$ and let $\bm{\Pi} \in \mathcal{P}^{d_{Y}}$ be a permutation matrix that corresponds to the permutation $\pi:[d_Y] \to [d_Y]$ defined as follows: $\pi(j) = 1,\ \pi(1) = j,\ \pi(k) = k,\ \text{for all } k \in [d_Y] \setminus \{1, j\}.$
According to \eqref{set:newb}, it is straightforward to verify $(\bm{\sigma}, \bm{\Pi}) \in \mathcal{B}$.
\end{proof}
Now, we present a condition under which Problem \eqref{eq:G} has a non-strict saddle point.
\begin{proposition}\label{prop:gamma3}
Suppose that \( L \geq 3 \) and $\bm Y$ is defined in \eqref{eq:Y}. Let \( (\bm{\sigma}, \bm{\Pi}) \in \mathcal{B} \) and $\bm{W} = (\bm{W}_1, \ldots, \bm{W}_L)$ $\in \mathcal{W}_{(\bm{\sigma}, \bm{\Pi})}$ be arbitrary and $r_{\sigma} := \|\bm \sigma\|_0$. If we have \( (y_{\pi(1)}, \ldots, y_{\pi(r_\sigma)}) = (y_1, \ldots, y_{r_\sigma}) \) and there exists \( i \in [r_\sigma] \) such that \( \sigma_i \in \mathcal{S}_3 \), then \( \bm{W} \) is a non-strict saddle point.
\end{proposition}
\begin{proof}
Consider that \( \bm{W} \) takes the form in \eqref{eq:simplification}.
{ Using \Cref{lem:derprop}(iv) and (v), we know that $\mathcal{S}_3 = \{x_*\}$, where $x_*$ is defined in \eqref{eq:x*}, and $z_1 > x_* > z_2$ for any $z_1 \in \mathcal{S}_1$ and $z_2 \in \mathcal{S}_2$.
Using this and the fact that $\bm{\sigma}$ is sorted in nonincreasing order, we conclude that there exist indices $p$ and $q$ satisfying $0 \le p < i \le q \le r_\sigma$ and
\[
\sigma_j \in \mathcal{S}_1 \quad \text{for } j = 1,\dots,p, \quad
\sigma_{p+1} = \cdots = \sigma_q \in \mathcal{S}_3, \quad
\sigma_j \in \mathcal{S}_2 \quad \text{for } j = q+1,\dots,r_\sigma.
\]
We now claim that $q = r_\sigma$, i.e., no element of $\bm{\sigma}$ belongs to $\mathcal{S}_2$. Suppose, for the sake of contradiction, that this is not the case.  Then $\sigma_{r_\sigma} \in \mathcal{S}_2$.
Since $(y_{\pi(1)},\ldots,y_{\pi(r_\sigma)}) = (y_1,\ldots,y_{r_\sigma})$ and $(\bm{\sigma},\bm{\Pi}) \in \mathcal{B}$ (see \eqref{set:newb}), we have
\[
f(\sigma_i; y_i) = f(\sigma_{r_\sigma}; y_{r_\sigma}) = 0.
\]
However, by \Cref{lem:derprop}(i) and (iv), $\sigma_i \in \mathcal{S}_3$ implies $y_i = y_*$. By \Cref{lem:derprop}(i), $\sigma_{r_\sigma} \in \mathcal{S}_2$ implies that $f(x; y_{r_\sigma}) = 0$ has two distinct positive roots and thus $y_{r_\sigma} > y_*$, which, together with $y_i = y_*$, contradicts $y_i \ge y_{r_\sigma}$. Hence, we must have $q = r_\sigma$.
Consequently, we obtain that there exists $0\le p < i\le r_\sigma$ such that  the sequence is partitioned as
\[\sigma_j \in \mathcal{S}_1 \ \text{for } j = 1,\dots,p, \quad
\sigma_{p+1} = \cdots = \sigma_{r_\sigma} \in \mathcal{S}_3.\footnote{  Note that when $p=0$, this reduces to $\sigma_1=\cdots=\sigma_{r_\sigma} \in \mathcal{S}_3$.} \]
}
    Now, we construct a direction $\bm \Delta = \left(\bm\Delta_1 \blk(\bm \Pi, \bm I), \bm\Delta_2, \dots, \bm\Delta_{L-1},\blk(\bm \Pi^T\!\!, \bm I)\bm \Delta_L\right)$, where $ (\bm{\Delta}_1, \ldots, \bm{\Delta}_L) \in \R^{d_1\times d_0} \times \dots \times \R^{d_L\times d_{L-1}}$ and $\|\bm\Delta_l\|\le 1$ for each $l \in [L]$.
    Let $\widehat{\bm W} := \bm W+t\bm \Delta$, which implies
    \begin{subequations}
    \begin{align}\notag
    & \widehat{\bmw}_1 = (\bm \Sigma_1+t\bm\Delta_1) \blk(\bm \Pi, \bm I),\ \widehat{\bm W}_L = \blk(\bm \Pi^T, \bm I)(\bm \Sigma_L+t\bm\Delta_L),\\
    & \widehat{\bm W}_l = \bm\Sigma_l+t\bm\Delta_l,\ l = 2,\dots,L-1. \notag
    \end{align}
    \end{subequations}
    For ease of exposition, we define $\widetilde{\bm W} : = (\bm \Sigma_L +t \bm \Delta_L)\cdots(\bm\Sigma_1+t\bm\Delta_1)$. Using a similar computation in \eqref{eq1:prop l3local} and \eqref{eq2:prop l3local}, we obtain
    \begin{align}\label{eq3:prop gamma3}
        &G(\widehat{\bmw})-G(\bmw)
       \ge  \sum_{i=1}^{d_Y} g\left(\sigma_i(\widetilde{\bmw})^{\frac{1}{L}};y_i\right) - \sum_{i=1}^{r_{\sigma}}g(\sigma_i;y_i
        ) - \sum_{i= r_{\sigma}+1}^{d_Y} g(0;y_i).
    \end{align}
By \Cref{coro:completelocal}, $\sigma_i$ is a local minimizer of $g(x; y_i)$ for all $i \in [p]$, and $0$ is a local minimizer of $g(x;y_i)$ for all $i \in [d_Y]$. For sufficiently small $t$ such that $t\bm \Delta_l$ meets the condition in \eqref{eq:delta_condition} for all $l \in [L]$, noting that $p\le r_\sigma$, we have
\begin{align*}
    g(\sigma_i(\widetilde{\bmw})^{\frac{1}{L}};y_i) \ge g(\sigma_i;y_i),\ \forall i \in [p];\ g(\sigma_i(\widetilde{\bmw})^{\frac{1}{L}};y_i)\ge g(0;y_i),\ \forall i\in \{r_\sigma+1,\dots,d_Y\}.
\end{align*}
This, together with \eqref{eq3:prop gamma3}, implies
\begin{align}
\notag G(\widehat{\bmw})-G(\bmw)
&\ge \sum_{i=1}^{r_\sigma}\left(g(\sigma_i(\widetilde{\bmw})^{\frac{1}{L}};y_i)-g(\sigma_i;y_i)\right)
\!+\!\!\sum_{i=r_{\sigma}+1}^{d_{Y}}\!\!\left(g(\sigma_i(\widetilde{\bmw})^{\frac{1}{L}};y_i)-g(0;y_i)\right)\\
&\ge \sum_{i=p+1}^{r_{\sigma}}\left(g(\sigma_i(\widetilde{\bmw})^{\frac{1}{L}};y_i)-g(\sigma_i;y_i)\right).
\notag
    \end{align}
For ease of exposition, we define \( \tau_i := \sigma_i(\widetilde{\bm W}) - \sigma_i^L \) for each \( i \in \{p+1, \dots, r_{\sigma}\} \).
Note that \eqref{eq:sigma_inequality} still holds and thus \( \tau_i \) can be arbitrarily small as $t\to 0$ for each such \( i \).
 According to \eqref{eq:defg}, for all $i \in \{p+1,\dots,r_{\sigma}\}$ and sufficiently small $t>0$, we have
    \begin{align}
&\ g(\sigma_i(\widetilde{\bmw})^{\frac{1}{L}};y_i)-g(\sigma_i;y_i) \notag\\
        =&\ \left(\sigma_i(\widetilde{\bm W})-\sqrt{\lambda}y_i\right)^2 +\lambda L \sigma_i^{{2}/{L}}(\widetilde{\bm W}) - \left(\sigma_i^L - \sqrt{\lambda}y_i\right)^2 - \lambda L \sigma_i^2 \notag\\
        =&\ \left(\tau_i +\sigma_i^L - \sqrt{\lambda}y_i\right)^2 +\lambda L \left(\tau_i +\sigma_i^L\right)^{{2}/{L}}- \left(\sigma_i^L - \sqrt{\lambda}y_i\right)^2 - \lambda L \sigma_i^2\notag\\
        =&\ \tau_i^2+2\tau_i(\sigma_i^L - \sqrt{\lambda}y_i) +\lambda L \sigma_i^2\left(\left(1 + \tfrac{\tau_i}{\sigma_i^L}\right)^{{2}/{L}}-1\right)\notag\\
        \ge&\
        \tau_i^2+2\tau_i(\sigma_i^L - \sqrt{\lambda}y_i) +\lambda L \sigma_i^2\left(\tfrac{2\tau_i}{L\sigma_i^L}+\tfrac{2-L}{L^2}\left(\tfrac{\tau_i}{\sigma_i^L}\right)^2-\tfrac{16(L-2)(L-1)}{3L^3}\left|\tfrac{\tau_i}{\sigma_i^L}\right|^3\right)\notag\\
        =&\ 2\tau_i(\sigma^L_i - \sqrt{\lambda}y_i+\lambda \sigma_{i}^{2-L}) +\tau_i^2\left(1+\tfrac{\lambda(2-L)}{L\sigma_i^{2L -2}}\right)-\tfrac{16(L-2)(L-1)\lambda}{3L^2 \sigma_i^{3L-2}}|\tau_i|^3,\label{eq:getp+1g}
    \end{align}
    where the inequality follows from \Cref{lem:inequality_alpha}. Note that as $\sigma_i\in \mathcal S_3$ for $i=p+1,\dots,r_\sigma$, \eqref{eq:f=0} and \eqref{eq:fprime=0} hold.
   Eliminating $y_{\pi(i)}$ in  \eqref{eq:f=0} and \eqref{eq:fprime=0}, we have $ L\sigma_i^{2L-1} =\lambda (L-2)\sigma_i,\ \text{for each}\  i \in \{p+1,\dots,r_{\sigma}\}$.
Using $y_i=y_{\pi(i)},~i=1,\dots,r_\sigma$ and \eqref{eq:f=0}, we obtain
\begin{equation}\label{eq:yip}
\sigma^L_i - \sqrt{\lambda}y_i+\lambda \sigma_{i}^{2-L}=0,~i=p+1,\dots,r_\sigma.
\end{equation}
Using the above two facts and \eqref{eq:getp+1g}, we have
\begin{align}\label{eq:finalpart}
    G(\widehat{\bmw})-G(\bmw) \ge - \sum_{i=p+1}^{r_\sigma}\frac{ 16(L-2)(L-1)\lambda}{3L^2 \sigma_i^{3L-2}}|\tau_i|^3.
    \end{align}
Using \eqref{eq:sigma_inequality} and $\|\bm{\Delta}_l\| \le 1$ for all $l \in [L]$, we immediately have $|\tau_i|\le L(\sigma_1+1)^{L-1}|t|$ for all $ i=p+1,\ldots,r_{\sigma}$.
   Substituting this into \eqref{eq:finalpart} yields
    \begin{align}
        \label{eq:trilowbound}
           G(\widehat{\bmw}) - G(\bmw)\! \ge\! -\!\! \sum_{i=p+1}^{r_\sigma}\frac{16(L-2)(L-1)\lambda L(\sigma_1+1)^{3L-3}}{3 \sigma_i^{3L-2}}|t|^3.
    \end{align}
    We claim that $\bm{W}$ is a second-order critical point. Indeed, if this were not the case, then by \Cref{def:crit}, there would exist a direction $\bm{\Delta}$ and sufficiently small $t > 0$ such that
\[
G(\bm{W} + t \bm{\Delta}) - G(\bm{W}) \le -c t^2
\]
for some constant $c > 0$, which contradicts \eqref{eq:trilowbound}.

To show that \( \bm W \) is a non-strict saddle point, we now construct a direction \( \bm \Delta = (\bm \Delta_1, \ldots, \bm \Delta_L)  \in \R^{d_1 \times d_0} \times \dots \times \R^{d_L \times d_{L-1}} \), where the $(m,n)$-th entry of $\bm \Delta_l$ is $1$ if $(m,n) = (p+1, p+1)$ and $0$ otherwise.
Thus, we have
\begin{align*}
&G(\widehat{\bmw})-G(\bmw)\\
        =&\left\|\prod_{l=L}^1(\bm \Sigma_l +t \bm \Delta_l)\! -\!\sqrt{\lambda}\widehat{\bm Y}\right\|_F^2 + \lambda \sum_{i=1}^{L}\|\bm\Sigma_i + t\bm\Delta_i\|_F^2 -\|\!\prod_{l=L}^1\bm \Sigma_l -\sqrt{\lambda}\widehat{\bm Y}\|_F^2 -\!\!\ \lambda \sum_{i=1}^{L}\|\bm\Sigma_i \|_F^2\\
        =&\left(\!(\sigma_{p+1}+t)^L\! -\! \sqrt{\lambda}y_{p+1}\right)^2 + \lambda L (\sigma_{p+1}+t)^2 -(\sigma_{p+1}^L\!-\! \sqrt{\lambda}y_{p+1})^2 -\lambda L\sigma_{p+1}^2 \\
        =& 2L(\!\sigma_{p+1}^{2L-1}\!\! -\!\sqrt{\lambda}y_{p+1}\sigma_{p+1}^{L-1} +\lambda \sigma_{p+1}) t\!+\!L\!\left(\!(2L-1)\sigma_{p+1}^{2L-2}-\!\sqrt{\lambda}(L-1)y_{p+1}\sigma^{L -2}_{p+1}\!+\!\lambda\right)t^2\\
        &\quad + \frac{L}{3}\left((2L-1)(2L-2)\sigma^{2L-3}_{p+1} -\sqrt{\lambda}(L-1)(L-2)y_{p+1} \retwo{\sigma^{L-3}_{p+1}} \right)t^3+ O(t^4)\\
        =&\ \frac{L}{3}\left((2L-1)(2L-2)\sigma^{2L-3}_{p+1} -\sqrt{\lambda}(L-1)(L-2)y_{p+1} \retwo{\sigma^{L-3}_{p+1}} \right)t^3 + O(t^4),
\end{align*}
where the last equality uses \eqref{eq:yip} and \eqref{eq:fprime=0} with $y_{\pi(p+1)} = y_{p+1}$. By {\Cref{lem:derprop}(iv)} and \eqref{eq:2der}, we have
\[
    \frac{\partial^2 f(\sigma_{p+1};y_{p+1})}{\partial x^2} = (2L-1)(2L-2)\sigma_{p+1}^{2L-3} - \sqrt{\lambda}(L-1)(L-2)y_{p+1} \sigma_{p+1}^{L-3} >0.
\]
When $t > 0$ (resp., $t < 0$) is sufficiently small, we have $G(\widehat{\bmw}) - G(\bmw)  > 0$ (resp., $G(\widehat{\bmw}) - G(\bmw)  < 0$).
Hence, $\bm{W}$ admits both descent and ascent directions and is thus a non-strict saddle point by \Cref{def:crit}.
\end{proof}
Using Lemma~\ref{lem:nonlocalmax}, Lemma~\ref{lem:illcase}, and Propositions~\ref{prop:gamma2}--\ref{prop:l3local} and~\ref{prop:gamma3}, we obtain the following corollary.
\begin{corollary}\label{coro:min}
    Suppose that $L \ge 3$ and $\bm Y$ is defined in \eqref{eq:Y}. Let $(\bm{\sigma}, \bm{\Pi}) \in \mathcal{B}$ be arbitrary and $\bm{W} = (\bm{W}_1, \ldots, \bm{W}_L) \in \mathcal{W}_{(\bm{\sigma}, \bm{\Pi})}$ be a critical point, and define $r_\sigma := \|\bm{\sigma}\|_0$.  Then $\bm{W}$ is a local minimizer if and only if $\sigma_i \in \mathcal{S}_1$ for all $i \in [r_{\sigma}]$ and $(y_{\pi(1)}, \ldots, y_{\pi(r_\sigma)}) = (y_1, \ldots, y_{r_\sigma})$.
\end{corollary}
Using \eqref{eq:defg} and the definition of $\mathcal{H}$ in \eqref{eq:defopt}, one verifies that for each $\bm{\sigma} \in \mathcal{H}$,
\begin{align}\label{eq:partial}
  2L\, f(\sigma_i; y_i) = \partial_x g(\sigma_i; y_i) = 0,\ \forall i \in [d_{\min}].
\end{align}
This condition allows us to characterize the global minimizers of Problem~\eqref{eq:G}.
\begin{proposition}\label{prop:opt}
 Suppose that $L \ge 3$ and $\bm Y$ is defined in \eqref{eq:Y}.  Let $(\bm{\sigma}, \bm{\Pi}) \in \mathcal{B}$ be arbitrary and let $\bm{W} = (\bm{W}_1, \ldots, \bm{W}_L) \in \mathcal{W}_{(\bm{\sigma}, \bm{\Pi})}$ be any critical point, and define $r_\sigma := \|\bm{\sigma}\|_0$.
   Then $\bm W$ is a global minimizer of Problem \eqref{eq:G} if and only if $\bm \sigma \in \mathcal{H}$.
\end{proposition}
\begin{proof}
Suppose that $\bm W \in \mathcal{W}_{(\bm{\sigma}, \bm{\Pi})}$ is a global minimizer. According to \Cref{coro:min}, we have $\sigma_i \in \mathcal{S}_1$ for each $i \in [r_{\sigma}]$ and $(y_{\pi(1)}, \ldots, y_{\pi(r_\sigma)}) = (y_1, \ldots, y_{r_\sigma})$.
Let $\bm{\sigma}^* \in \mathcal{H}$ be arbitrary.
Thus,   using \eqref{eq:simplification}, \eqref{eq:Yhat}, and \eqref{eq:calculation}, we have
\begin{align}\label{eq:subf}
    G(\bmw) &= \sum_{i =1}^{r_{\sigma}} \left((\sigma_i^L - \sqrt{\lambda}y_{\pi(i)})^2+ L  \lambda\sigma_i
    ^2\right)+\lambda\sum_{i= r_{\sigma}+1}^{d_{Y}}  y_{\pi(i)}^2\notag \\
    &=\sum_{i =1}^{r_{\sigma}} \left((\sigma_i^L - \sqrt{\lambda}y_{i
    })^2+ L  \lambda\sigma_i
    ^2\right)+\lambda\sum_{i= r_{\sigma}+1}^{d_{Y}}  y_{i}^2 \notag \\
    &= \sum_{i=1}^{r_\sigma} g(\sigma_i;y_i) + \sum_{i = r_{\sigma}+1}^{d_{\min}}g(0;y_i)+\lambda\!\!\!\!\!\sum_{i=d_{\min}+1}^{d_Y}y_i^2 \ge \sum_{i=1}^{d_{\min}} g(\sigma^{*}_i;y_i) +\lambda\sum_{i=d_{\min}+1}^{d_Y}y_i^2,
\end{align}
where
the inequality follows from $\bm \sigma^* \in \mathcal{H}$. By \eqref{eq:partial}, $f(\sigma^{*}_i; y_i) = 0$ for all $i \in [d_{\min}]$. This, together with \eqref{set:newb}, gives $(\bm{\sigma}^*, \bm{I}_{d_Y}) \in \mathcal{B}$.
 For any $\bm{W}^* \in \mathcal{W}_{(\bm{\sigma}^*, \bm{I}_{d_Y})}$, we compute $G(\bm W^*) = \sum_{i=1}^{d_{\min}} g(\sigma^{*}_i; y_i) + \lambda\sum_{i=d_{\min}+1}^{d_Y} y_i^2.$
Since $\bm{W}$ is a global minimizer,
the inequality in \eqref{eq:subf} holds with equality.
Thus, we conclude $\bm {\sigma}\in \mathcal{H}$.

Now, suppose that $\bm \sigma \in \mathcal{H}$. We have
\begin{align*}
    &f(\sigma_i; y_{\pi(i)})
     = \sigma_i^{2L -1} - \sqrt{\lambda}y_{\pi(i)} \sigma_i^{L-1} +\lambda \sigma_i
    \overset{\eqref{set:newb}}{=} 0,\
    \ \forall i \in [r_{\sigma}], \\
    &\partial_xg(\sigma_i ; y_i)
    = 2L(\sigma_i^{2L -1} - \sqrt{\lambda}y_{i} \sigma_i^{L-1} +\lambda \sigma_i) \overset{\eqref{eq:partial}}{=}  0,\ \forall i \in [r_\sigma].
\end{align*}
Since $\sigma_i > 0$ for all $i \in [r_\sigma]$, it follows that $y_i = y_{\pi(i)}$ for all $i \in [r_\sigma]$. This yields
\begin{align}\label{eq1:prop opt}
    G(\bm{W}) = \sum_{i=1}^{d_{\min}} g(\sigma_i; y_i) + \lambda \sum_{i=d_{\min}+1}^{d_Y} y_i^2.
\end{align}
Now, it suffices to show $G(\widehat{\bm{W}}) \ge G(\bm{W})$ for any local minimizer $\widehat{\bm{W}} \in \mathcal{W}_{(\hat{\bm{\sigma}}, \hat{\bm{\Pi}})}$, where $(\hat{\bm{\sigma}}, \hat{\bm{\Pi}}) \in \mathcal{B}$. According to \Cref{coro:min}, we obtain that $\hat{\sigma}_i \in \mathcal{S}_1$ for all $i \in [r_{\hat{\sigma}}]$ and that the permutation $\hat{\pi}$ satisfies $(y_{\hat{\pi}(1)}, \ldots, y_{\hat{\pi}(r_{\hat{\sigma}})}) = (y_1, \ldots, y_{r_{\hat{\sigma}}})$. Using the same computation in \eqref{eq:subf}, we obtain $G(\widehat{\bm{W}}) = \sum_{i=1}^{d_{\min}} g(\hat{\sigma}_i; y_i) + \lambda \sum_{i=d_{\min}+1}^{d_Y} y_i^2.$ This, together with \eqref{eq1:prop opt} and the definition of $\mathcal{H}$, we have $G(\widehat{\bm{W}}) \ge G(\bm{W})$.
\end{proof}
Armed with the above setup, we are ready to prove \Cref{thm:land} and \Cref{coro:land 1}.
\begin{proof}[Proof of \Cref{thm:land}]
By \Cref{lem:equi FG}(i) and (iii), it suffices to analyze a critical point of Problem~\eqref{eq:G}, i.e., \( (\sqrt{\lambda_1} \bm{W}_1, \ldots, \sqrt{\lambda_L} \bm{W}_L) \). This, together with \Cref{prop:opti G}, yields that \( (\sqrt{\lambda_1} \bm{W}_1, \ldots, \sqrt{\lambda_L} \bm{W}_L) \in \mathcal{W}_{(\bm{\sigma}, \bm{\Pi})} \). It follows from \Cref{prop:equivalent} that all critical points in \( \mathcal{W}_{(\bm{\sigma}, \bm{\Pi})} \) are of the same type. Using Propositions \ref{prop:gamma2} and \ref{prop:misalign}, we conclude that all points in \( \mathcal{W}_{(\bm{\sigma}, \bm{\Pi})} \) are strict saddle points if there exists \( i \in [r_\sigma] \) such that \( \sigma_i \in \mathcal{S}_2 \) or \( (y_{\pi(1)}, \ldots, y_{\pi(r_\sigma)}) \neq (y_1, \ldots, y_{r_\sigma}) \). As a result, we prove (i). Then, (ii), (iii), and (iv) follow directly from Propositions \ref{prop:l3local}, \ref{prop:opt}, and \ref{prop:gamma3}, respectively.
\end{proof}
\begin{proof}[Proof of \Cref{coro:land 1}]
Suppose that \eqref{eq:wellcase} holds. Now, we show that each critical point of Problem~\eqref{eq:F} is either a local minimizer or a strict saddle point.
By \Cref{lem:equi FG}(ii) and \Cref{lem:nonlocalmax}, this is equivalent to showing that every critical point
$\bm{W} = (\bm{W}_1, \ldots, \bm{W}_L)$ of Problem~\eqref{eq:G} is not a non-strict saddle point. Suppose that $\bm{W} \in \mathcal{W}_{(\bm{\sigma}, \bm{\Pi})}$ for some $(\bm{\sigma}, \bm{\Pi}) \in \mathcal{B}$.
If $\bm W$ were a non-strict saddle point, then \Cref{thm:land}(i),(ii) imply that
$(y_{\pi(1)},\ldots,y_{\pi(r_\sigma)})=(y_1,\ldots,y_{r_\sigma})$ and
$\sigma_k\in\mathcal S_3$ for some $k\in[r_\sigma]$. Hence
$k\le r_\sigma\le d_{\min}$ and $y_{\pi(k)}=y_k$. Since
$(\bm\sigma,\bm\Pi)\in\mathcal B$, we have $f(\sigma_k;y_{\pi(k)})=0$. Using $\sigma_k\in\mathcal S_3$ and \Cref{lem:derprop}(i),(iv), we obtain that the equality in \eqref{eq:wellcase} holds for $i=k$, contradicting \eqref{eq:wellcase}. Thus,
$\bm W$ is not a non-strict saddle point. According to \Cref{lem:nonlocalmax} and \Cref{def:crit}, each critical point is either a local minimizer or a strict saddle point.

Now suppose that \eqref{eq:wellcase} fails. By \Cref{lem:equi FG}(ii), this is equivalent to show the existence of a non-strict saddle point for Problem~\eqref{eq:G}.
Let \( p\in[d_{\min}] \) be the smallest index for which the equality in \eqref{eq:wellcase} holds. This, together with \Cref{lem:derprop}(iv), implies that \( \hat{x}(y_p) \in \mathcal{S}_3 \).
Let $\bm{\Pi} = \bm{I}_{d_Y}$ and $\bm{\sigma} := \left( \overline{x}(y_1), \ldots, \overline{x}(y_{p-1}), \hat{x}(y_p), 0, \ldots, 0 \right) \in \mathbb{R}^{d_{\min}}.$ By $y_1\ge y_2\ge \dots\ge y_{d_{\rm min}}$, { \Cref{lem:derprop}(ii) and (v)}, we have
\begin{align*}
    \overline{x}(y_1) \geq \cdots \geq \overline{x}(y_{p-1}) \geq \hat{x}(y_p) \geq 0 \quad \text{and} \quad f(\sigma_i; y_i) = 0,\ \forall i \in [d_{\min}].
\end{align*}
Thus, $r_\sigma=p$. Since $\bm{\Pi} = \bm{I}_{d_Y}$, then the associated permutation $\pi(\cdot)$ is the identity mapping, and thus $(y_{\pi(1)}, \ldots, y_{\pi(p)}) = (y_1, \ldots, y_p)$. By the definition of $\mathcal{B}$ in \eqref{set:newb}, we have $(\bm{\sigma}, \bm{\Pi}) \in \mathcal{B}$. By \Cref{prop:gamma3}, every $\bm{W} \in \mathcal{W}_{(\bm{\sigma}, \bm{\Pi})}$ is a non-strict saddle point of Problem~\eqref{eq:G}.
\end{proof}
{
\begin{remark}
    We remark that our proof techniques totally differ from those in \cite{achour2024loss} due to the presence of regularization. In the unregularized setting, the product $\bm{W}_L \cdots \bm{W}_1$ can be treated as a single matrix, and critical points are characterized via invariance under general invertible transformations. In contrast, the introduction of regularization removes this invariance with only orthogonal invariance remaining; moreover, it enforces the balancedness condition, i.e., $\bm W_l\bm W_l^T=\bm W_{l+1}^T\bm W_{l+1}$. Consequently, the rank-based arguments in \cite{achour2024loss} do not apply to the loss landscape analysis in the regularized case. Instead, in the regularized case, the alignment among singular vectors and the equation \eqref{eq:thmsimga pi} play a central role in our analysis. This shift invalidates the saddle point characterization based on the rank-deficiency (i.e., ``tightened pivots'')  used in \cite{achour2024loss}, necessitating entirely new proof techniques to characterize critical points.
\end{remark}}

\section{Experimental Results}\label{sec:expe}
In this section, we conduct numerical experiments to visualize the loss landscape of Problem \eqref{eq:F} in support of our theoretical findings. These visualizations provide intuitive insight into the geometry of different types of critical points, such as local minimizers and strict saddle points. To visualize the loss landscape of Problem \eqref{eq:F}, we use the random projection approach proposed in \cite{goodfellow2014qualitatively,li2018visualizing}. Note that the optimization variables of Problem \eqref{eq:F} are $L$ high-dimensional matrices, i.e., $\bm W=(\bm W_1,\dots,\bm W_L) \in \R^{d_1\times d_0}\times \dots \times\R^{d_L\times d_{L-1}}$. To enable visualization, we select a critical point $\bm{W}^*=(\bm W_1^*,\dots,\bm W_L^*)$ as the reference point. Next, we randomly generate two vectors of size $\sum_{l=1}^L{d_{l-1}d_l}$ with elements drawn from a standard Gaussian distribution and apply the Gram–Schmidt process to generate two vectors $\bm{\delta}_1 \in \R^{\sum_{l=1}^L{d_{l-1} d_l}}, \bm{\delta}_2 \in \R^{\sum_{l=1}^L{d_{l-1} d_l}}$, which are orthogonal to each other. Then, we reshape these two vectors into two matrices $\bm \Delta_1, \bm \Delta_2$, each of which has the same size as $\bm W^*$. Finally, we plot the following function
\begin{align*}
    h(\alpha,\beta) := G(\bm{W}^* + \alpha \bm{\Delta}_1 + \beta \bm{\Delta}_2),
\end{align*}
where $\alpha, \beta \in [-1,1]$.
In our experiments, we set the input and output dimensions $d_0 = 10$ and $d_L = 20$, respectively, and $d_l = 5$ for each hidden layer $l = 1, \dots, L-1$. For the data matrix $\bm Y \in \R^{d_L\times d_0}$, we independently sample each diagonal entry from a uniform distribution over $[0,10]$, i.e., $y_{ii} \overset{i.i.d.}{\sim} U(0,10)$, and set all off-diagonal entries to zero, i.e., $y_{ij} = 0$ for all $i \neq j$. Moreover, the regularization parameters are uniformly set to $\lambda_l = 10^{-4}$ for all $l \in [L]$.

\begin{figure}[t]
    \centering
    \begin{subfigure}[b]{0.3\textwidth}
    \includegraphics[width=\textwidth]{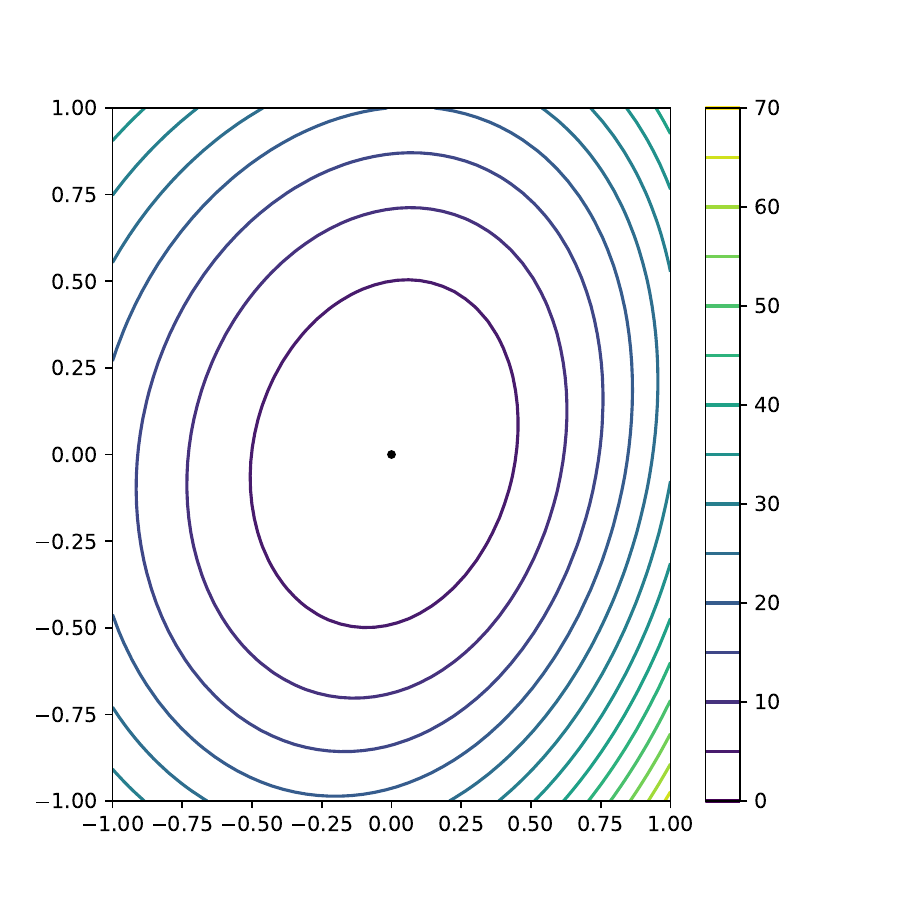}\vspace{-0.1in}
        \caption{Global minimizer (2D)}
        \label{fig:image1}
    \end{subfigure}
    \hfill
    \begin{subfigure}[b]{0.3\textwidth}
        \includegraphics[width=\textwidth]{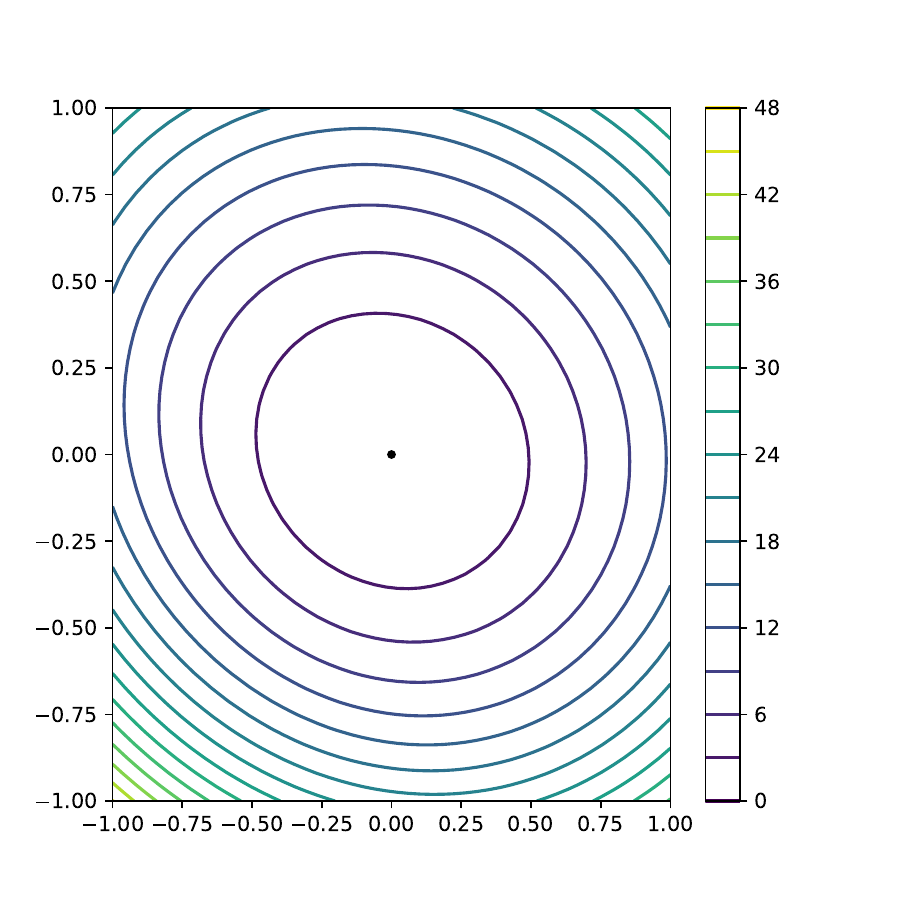}\vspace{-0.1in}
        \caption{Local minimizer (2D)}
        \label{fig:image2}
    \end{subfigure}
    \hfill
    \begin{subfigure}[b]{0.3\textwidth}
        \includegraphics[width=\textwidth]{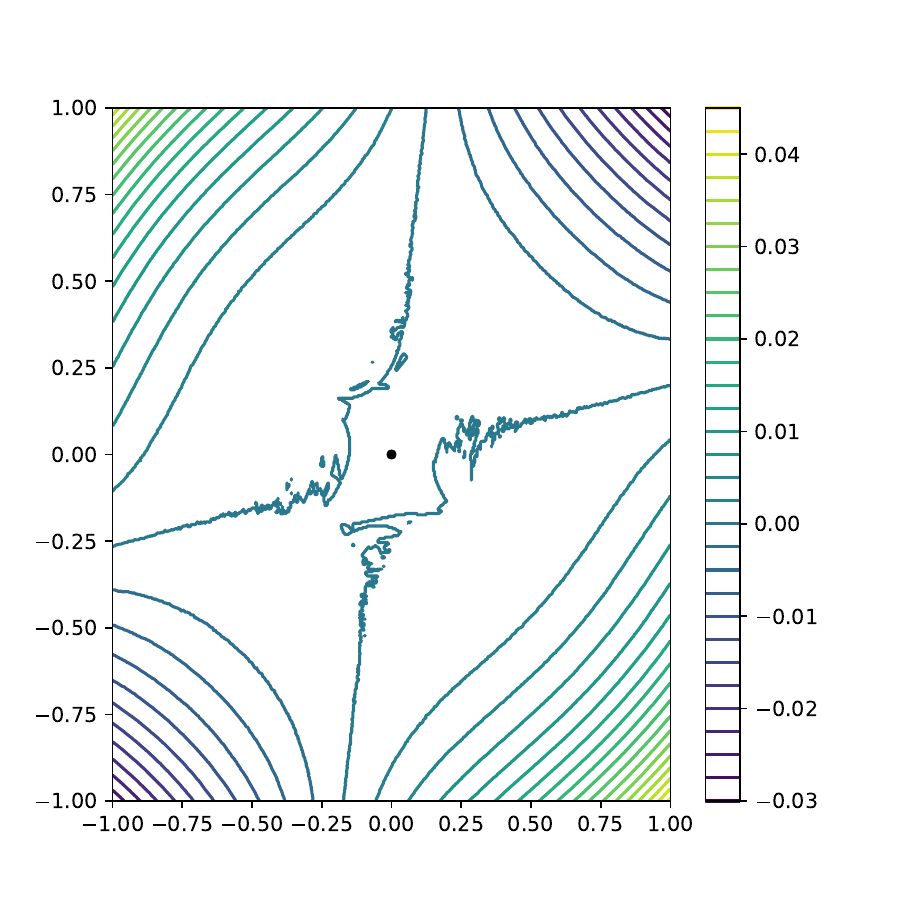}\vspace{-0.1in}
        \caption{Strict saddle (2D)}
        \label{fig:image3}
    \end{subfigure}

    \begin{subfigure}[b]{0.3\textwidth}
        \includegraphics[width=\textwidth]{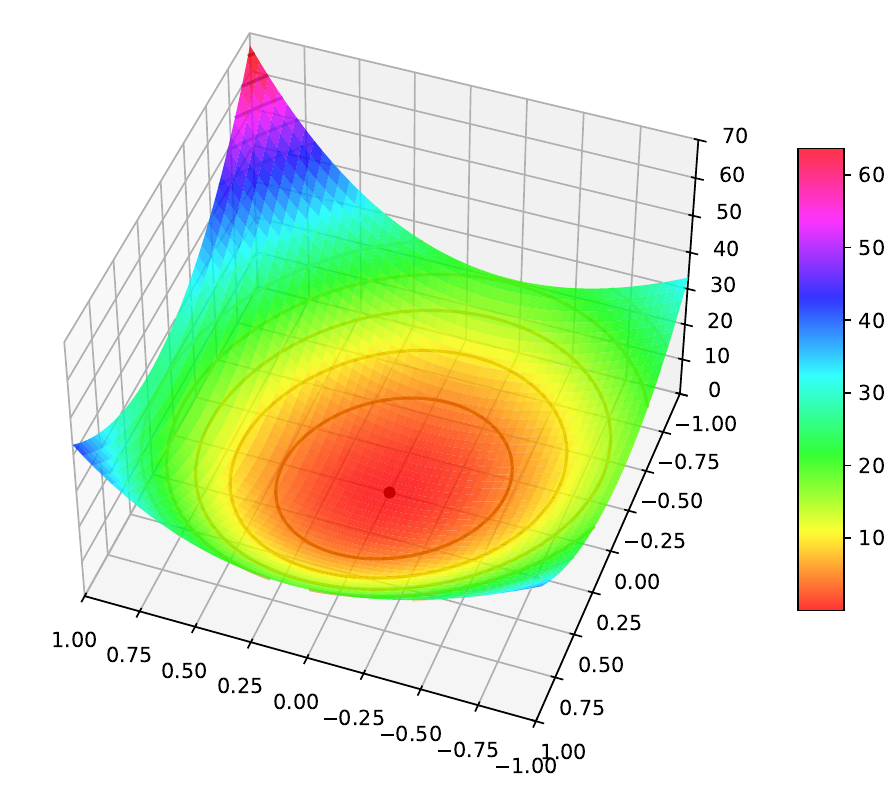}\vspace{-0.05in}
        \caption{Global minimizer (3D)}
        \label{fig:image4}
    \end{subfigure}
    \hfill
    \begin{subfigure}[b]{0.3\textwidth}
        \includegraphics[width=\textwidth]{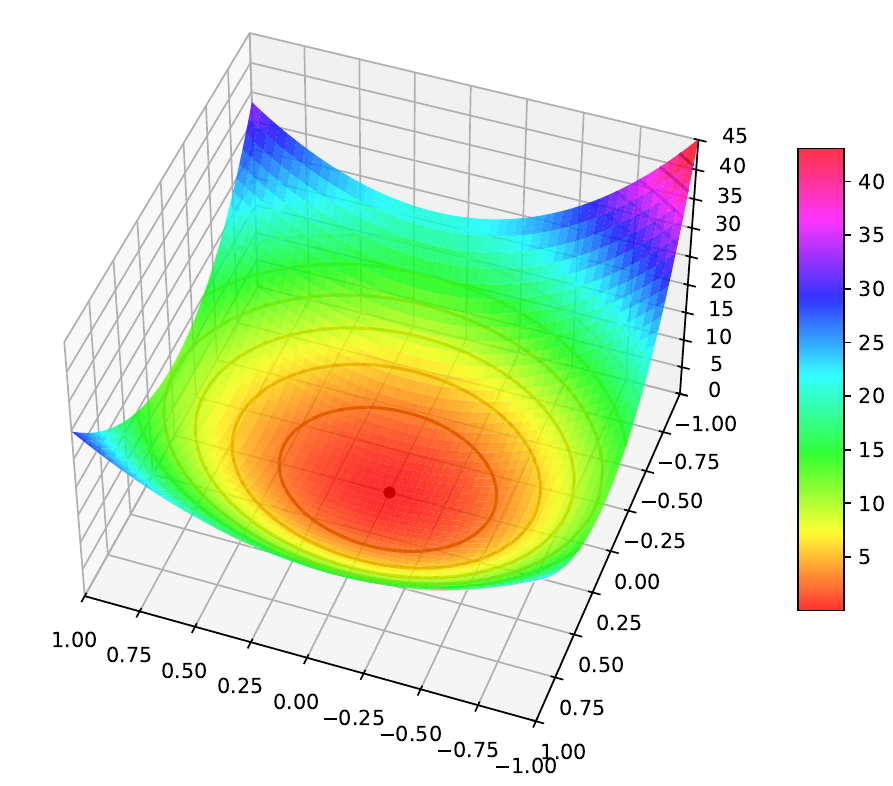}\vspace{-0.05in}
        \caption{Local minimizer (3D)}
        \label{fig:image5}
    \end{subfigure}
    \hfill
    \begin{subfigure}[b]{0.3\textwidth}
        \includegraphics[width=\textwidth]{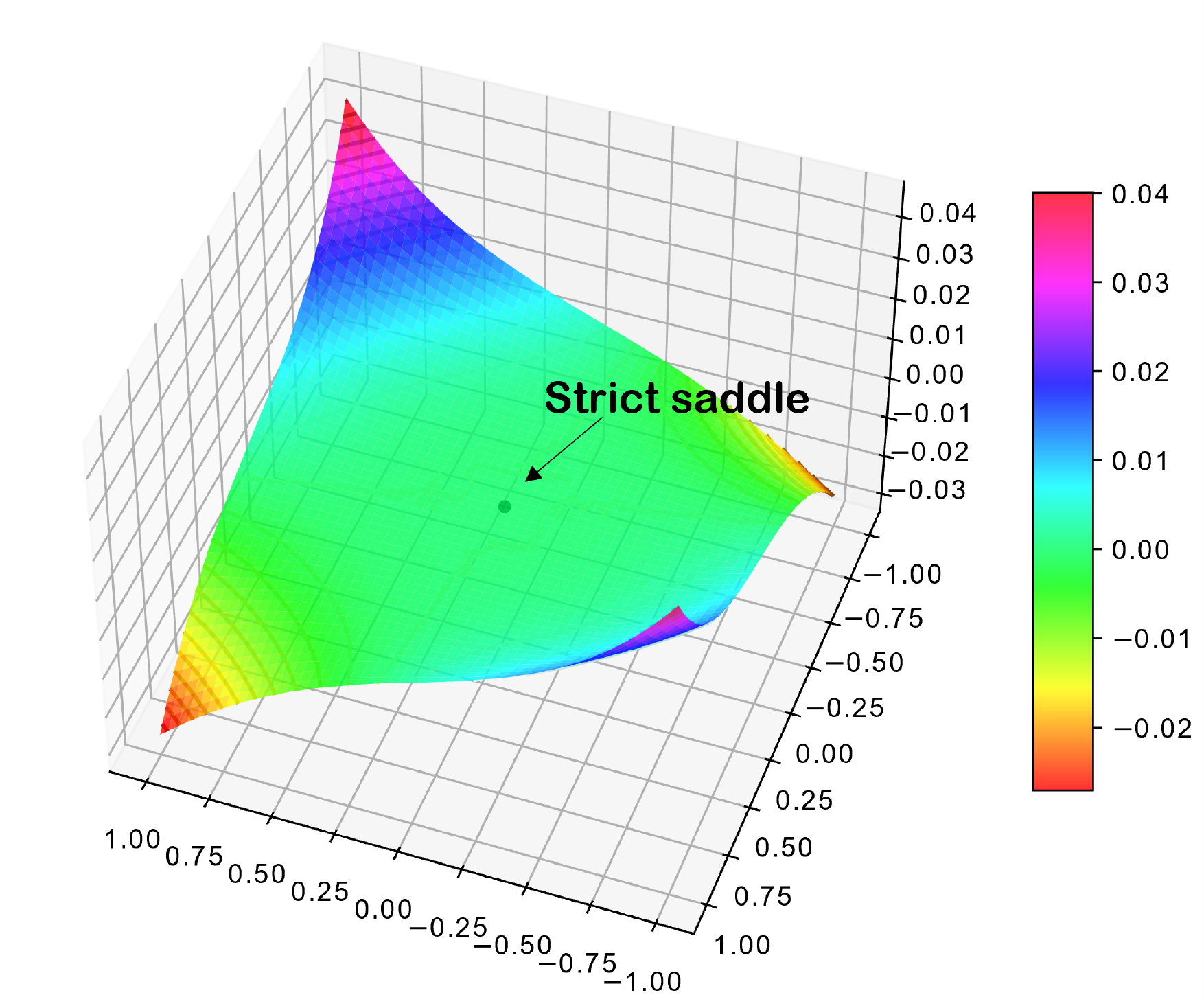}\vspace{-0.05in}
        \caption{Strict saddle (3D)}
        \label{fig:image6}
    \end{subfigure}
    \vspace{-0.15in}
    \caption{{\bf Visualization of the loss landscape of deep matrix factorization ($L = 4$).} The first row displays 2D contour plots for the global minimizer, local minimizer, and strict saddle point, respectively. The second row presents corresponding 3D contour plots.
    Note that the function value visualized in the figures is $h(\alpha,\beta) - h(0,0)$, where $h(0,0) = 39.47$ is the value at the global minimizer; the values at the local minimizer and strict saddle are $59.31$ and $289.07$, respectively.
    } \vspace{-0.2in}
    \label{fig:all_images}
\end{figure}

In the experiments, we visualize the loss landscape of Problem \eqref{eq:F} when $L = 4$. Specifically, we construct a global minimizer, a non-global local minimizer, and a strict saddle point according to  \Cref{thm:crit}, Propositions \ref{prop:misalign}, \ref{prop:l3local}, and \ref{prop:opt}, and use them as reference points.
Then, we visualize the loss landscape around them using 2D and 3D contour plots in \Cref{fig:all_images}. Note that the function values in the figures correspond to $h(\alpha,\beta) - h(0,0)$, highlighting the relative landscape around the reference point. We observe from \Cref{fig:all_images}(a,b,d,e) that the landscape around a local or global minimizer is a bowl-shaped structure, exhibiting a well-behaved region for optimization. In contrast, we observe from \Cref{fig:all_images}(c,f) that the landscape around a strict saddle point has directions of both ascent and descent, forming a saddle geometry.

\section{Conclusion}\label{sec:conc}
In this paper, we studied the loss landscape of the regularized deep matrix factorization problem \eqref{eq:F}. We first derived a closed-form characterization of the critical point set. Building on this characterization, we provided a complete loss landscape analysis by classifying all its critical points. Moreover, we gave a necessary and sufficient condition under which the loss function has a partially benign landscape in the sense that each critical point is either a local minimizer or a strict saddle point. This result provides a rigorous theoretical explanation for the empirical success of gradient-based optimization methods in deep matrix factorization tasks. {  As future work, and motivated by \cite{achour2024loss,kawaguchi2016deep,lu2017depth,nouiehed2022learning,Trager2020Pure}, we plan to extend our study to deep linear networks with general input data, moving beyond the orthogonal-input setting. Motivated further by recent advances on the loss landscapes of deep nonlinear networks \cite{du2018gradient,liu2022spurious,sun2020global}, a promising direction is to broaden our analysis to deep networks with nonlinear activation functions.
}
\bibliographystyle{siamplain}
\bibliography{references,NC}
\appendix
\section{Loss Landscape when $L=2$}\label{sec:app A}
We characterize the set of global minimizers and the loss landscape of Problem~\eqref{eq:G} when $L = 2$. Notably, these results are { adapted from} \cite[Lemma~B.2]{zhou2022optimization} and \cite[Lemma~4.2]{panshaohua2021}.
\begin{proposition} \label{prop:benign2}
For $L=2$, the following statements hold for Problem \eqref{eq:F}: \\
(i) Let
$
\bm{\sigma}^* :=  ((\sqrt{\lambda}y_1 - \lambda)^{\frac{1}{2}}_+,\ \ldots,\ (\sqrt{\lambda}y_{d_{\min}} - \lambda)^{\frac{1}{2}}_+),
$
where $x_+ := \max(x, 0)$ and $\lambda = \lambda_1\lambda_2$.\footnote{The result \cite[Lemma~B.2]{zhou2022optimization}  indicates that for the optimal solution $(\bm{W}_2, \bm{W}_1)$, the $i$-th singular value $\eta_i$ of the product $\bm W_2\bm W_1$ satisfies $\eta_i = (y_i - \sqrt{\lambda_1 \lambda_2})_{+}$.
This matches \Cref{prop:benign2} under  the relation $\eta_i = \sigma_i^{*2}/\sqrt{\lambda_1\lambda_2}$, since this relation together with $\eta_i = (y_i - \sqrt{\lambda_1\lambda_2})_{+}$ and $\lambda = \lambda_1\lambda_2$ yields $\sigma_i^* = (\sqrt{\lambda}\, y_i - \lambda)^{1/2}_{+}$.
} Then $(\bm{W}_1,\bm{W}_2)$ is an optimal solution
for Problem~\eqref{eq:F} {  if and only if}
\[
\bm{W}_1\! =\!\bm{Q}\blk\big( \mathrm{diag}(\bm{\sigma}^*)/\sqrt{\lambda_1},\, \bm{0} \big){  \bm V_Y^T},
\bm{W}_2\! =\!{  \bm U_Y} \blk\big( \mathrm{diag}(\bm{\sigma}^*)/\sqrt{\lambda_2},\, \bm{0} \big) \bm{Q}^T,
\]
{  where $\bm{Q} \in \mathcal{O}^{d_1}$ and $\bm{Y} = \bm{U}_Y \bm{\Sigma}_Y \bm{V}_Y^T$ is an SVD with $\bm{U}_Y \in \mathcal{O}^{d_2}$, $\bm{V}_Y \in \mathcal{O}^{d_0}$.}\\
(ii) Any critical point $\bm{W} \in \mathcal{W}_F$ is either a global minimizer or a strict saddle point.
\end{proposition}

\section{Discussion on the Function about Critical Points}\label{subsec:function}

Here, we study the number of non-negative roots of the equation $f(x;y) = 0$ in terms of $x$, where $f(x;y)$ is defined in \eqref{eq:deffty}.
For any fixed value of \( y\ge 0 \), we compute the derivatives of \( f(x; y) \) with respect to $x$ as follows
\begin{subequations}
\begin{align}
    \frac{\partial f(x; y)}{\partial x} &= (2L - 1)x^{2L - 2} - \sqrt{\lambda}(L - 1)y x^{L - 2} + \lambda, \label{eq:1der}\\
    \frac{\partial^2 f(x; y)}{\partial x^2} &= (2L - 1)(2L - 2)x^{2L - 3} - \sqrt{\lambda}(L - 1)(L - 2)y x^{L - 3}. \label{eq:2der}
\end{align}
\end{subequations}
Obviously, $x = 0$ is a root for any $y \ge 0$.
For $L \ge 3$, one can show that the equation $\partial_{xx} f(x; y) = 0$ has at most one positive root. Together with $\partial_x f(0; y) = \lambda > 0$, this implies that $\partial_x f(x; y) = 0$ has at most two distinct positive roots. Combining this with the boundary conditions $f(0; y) = 0$ and $\lim_{x \to +\infty} f(x; y) = +\infty$, we conclude that $f(x; y) = 0$ has at most two distinct positive roots for any fixed $y \ge 0$.

\subsection{Proof of \Cref{lem:derprop}}\label{subsec:proofderprop}
\begin{proof}
(i)
Given $y \ge 0$, computing the positive roots of \(f(x;y)=0\) is equivalent to computing the positive roots of $v(x)-y = 0$, where
\[
v(x):=\frac{x^{2L-1}+\lambda x}{\sqrt{\lambda}\,x^{L-1}} .
\]
Clearly, \(v(x)\to+\infty\) as \(x\to0_+\) and \(v(x)\to+\infty\) as \(x\to+\infty\).
Differentiating $v(x)$ and solving
\[
v'(x)= \frac{L}{\sqrt{\lambda}}x^{L-1} - (L-2)\sqrt{\lambda}x^{1-L} = 0,
\]
yield the solution \(x_*\) as defined in \eqref{eq:x*}.
One can verify that \(v(x)\) decreases on \((0,x_*)\) and increases on \((x_*,\infty)\),
attaining its minimum at \(x_*\) with $v(x_*) = y_*$.
Therefore, \(v(x)=y\)  has two distinct positive roots if \(y > y_*\), a unique positive root if \(y = y_*\), and no positive root if \(y < y_*\).

(ii)
It follows from (i) and $y>y_*$ that $\underline{x}(y)\in(0,x_*)$  and $\overline{x}(y)\in(x_*,\infty)$. Differentiating $v(\, \underline{x}(y)\,)=y$ and $v(\,\overline{x}(y)\,)=y$ with respect to $y$ yields
\[
\underline{x}'(y)=\frac{1}{v'(\underline{x}(y))},\qquad
\overline{x}'(y)=\frac{1}{v'(\overline{x}(y))}.
\]
Since $v'(x)<0$ for $x\in(0,x_*)$ and $v'(x)>0$ for $x\in(x_*,\infty)$, we obtain
$
\underline{x}'(y)<0$ and $\overline{x}'(y)>0.$
Hence the smaller root $\underline{x}(y)$ is strictly decreasing in $y$, while the larger root $\overline{x}(y)$ is strictly increasing in $y$.

(iii) Let $y > y_* > 0$ be fixed. From (i), we obtain that $f(x; y) = 0$ has two distinct positive roots.
According to \eqref{eq:2der}, \( \partial_{xx} f(x; y) = 0\) has only one positive root, denoted by $\alpha$. This, together with \eqref{eq:1der}, implies that \( \partial_x f(x; y)\) is monotonically decreasing on the interval \( [0, \alpha] \) and monotonically increasing on \( (\alpha, +\infty) \). Now, we claim that $\partial_x f(x;y) = 0$ has two distinct positive roots, denoted by $\beta_1$ and $\beta_2$, where $\beta_1 < \beta_2$. Then, we prove the claim by contradiction. Suppose that $\partial_x f(x;y) = 0$ has either a unique positive root or no positive roots. This, together with \( \partial_x f(0; y)= \lambda > 0 \) and the monotonicity of  \( \partial_x f(x; y)\), implies that $\partial_x f(x;y) > 0$ for all $x \ge 0$ and $x \neq \alpha$. Using this and $f(0;y) = 0$, we obtain that $f(x;y) > 0$  for all $x > 0$, which implies $f(x;y) = 0$ has no positive roots. This contradicts the fact that $f(x;y) = 0$ has two positive roots.

Note that \( 0 < \beta_1 < \alpha < \beta_2 \). One can verify that \( f(x; y) \) is monotonically increasing on \( [0, \beta_1] \), monotonically decreasing on \( (\beta_1, \beta_2] \), and monotonically increasing on \( (\beta_2, +\infty) \). Noting that \(  \overline{x}(y) > \underline{x}(y) > 0\) denote the positive roots of \( f(x;y) = 0 \), we have that \( \underline{x}(y) \in (\beta_1, \beta_2) \) and \( \overline{x}(y) \in (\beta_2, +\infty) \). This implies the desired results.

(iv)
According to the proof of (i), $f(x; y) = 0$ with $y \ge 0$ has a unique positive root if and only if $y=y_*$.
As discussed in (i), this root is just $x_*$.
Therefore, when $\mathcal{S}_3 \neq \emptyset$, we have $\mathcal{S}_3 = \{x_*\}$.
Substituting the expressions of $x_*$ and $y_*$ into \eqref{eq:1der} and \eqref{eq:2der},
we obtain $\tfrac{\partial f(x_*; y_*)}{\partial x} = 0$ and $\tfrac{\partial^2 f(x_*; y_*)}{\partial x^2}=2(L-1)(L-2)\left(\frac{L}{L-2}\right)^{\frac{1}{2L-2}}\lambda^{\frac{2L-3}{2L-2}} > 0$.

(v) By (i), $v(x)=y_*$ has a unique solution $\hat x(y_*)=x_*$. Hence, if $\mathcal S_3\neq\emptyset$, we have $\mathcal S_3=\{\,x_*\,\}$. Moreover, for any $z_1\in\mathcal S_1$ and $z_2\in\mathcal S_2$, it follows from the proof in (i) that $z_1>x_*$ and $z_2<x_*$. This directly implies
$
z_1 > x_* > z_2
$.
\end{proof}

\section{Auxiliary Results}
\begin{lemma}[Mirsky Inequality \cite{stewart1990matrix}]\label{lem:mirsky}
For any matrices $\bm X, \tilde{\bm X} \in \mathbb{R}^{m \times n}$ with singular values
$
\sigma_1 \geq \sigma_2 \geq \cdots \geq \sigma_{l}$ and $ \tilde{\sigma}_1 \geq \tilde{\sigma}_2 \geq \cdots \geq \tilde{\sigma}_{l},
$
where $l=\min\{m,n\}$,
then for any unitarily invariant norm (e.g., $\|\cdot\|_F$), we have $$\|\tilde{\bm X} - \bm X\| \ge \|\mathrm{diag}(\tilde{\sigma}_1 - \sigma_1,\dots,\tilde{\sigma}_{l} - \sigma_{l})\|.$$
\end{lemma}
\begin{lemma}\label{lem:schattenp}
{\em (\cite[Corollary 1]{xu2017unified})} Given $L \ (L \geq 2)$ matrices $\{\bm X_l\}_{l=1}^L$, where $\bm X_1 \in \mathbb{R}^{m \times d_1}$, $\bm X_l \in \mathbb{R}^{d_{l-1} \times d_l}$, $l = 2,\ldots,L-1$, $\bm X_L \in \mathbb{R}^{d_{L-1} \times n}$, and $\bm X \in \mathbb{R}^{m \times n}$ with $\mathrm{rank}(\bm X)=r \leq \min\{d_l,l=1,\ldots,L\}$, it holds for any $p,p_1,\ldots,p_L > 0$ satisfying $1/p = \sum_{l=1}^L {1}/{p_l}$ that
\begin{equation*}
\frac{1}{p}\|\bm X\|_{S_p}^p = \min_{\bm X_l:\bm X = \prod_{l=1}^L \bm X_l} \sum_{l=1}^L \frac{1}{p_l}\|\bm X_l\|_{S_{p_l}}^{p_l},
\end{equation*}
where the Schatten-$p$ norm
is defined as $\| \bm X \|_{S_p} := \left( \sum_{i=1}^{\min(m,n)} \sigma_i^p(\bm X) \right)^{{1}/{p}}.$
\end{lemma}
\begin{lemma}\cite[Proposition 9.6]{boumal2023introduction}\label{lem:boumalcite}
 Let $\mathcal{M}_1$ and $\mathcal{M}_2$ be two manifolds and $\phi\colon \mathcal{M}_2 \to \mathcal{M}_1$ be a map. Consider a function $f_1\colon \mathcal{M}_1 \to \mathbb{R}$ and its lift $f_2 = f_1\circ \phi\colon \mathcal{M}_2 \to \mathbb{R}$. The optimization problems $\min_{\bm x \in \mathcal{M}_1} f_1(\bm x)$ and $\min_{\bm x \in \mathcal{M}_2} f_2(\bm x)$ are related as follows:\\
 (i) If $\phi$ is surjective, $\bm x$ is a global minimizer of $f_2$ if and only if $\phi(\bm x)$ is a global minimizer of $f_1$.\\
 (ii) If $\phi$ is continuous and open, $\bm x$ is a local minimizer of $f_2$ if and only if $\phi(\bm x)$ is a local minimizer of $f_1$. \\
(iii) If $\phi$ is smooth and its differential at each point is surjective,  then $\bm x$ is a first-order (resp., second-order) critical point of $f_2$ if and only if $\phi(\bm x)$ is a first-order (resp., second-order) critical point of $f_1$.
\end{lemma}

\begin{lemma}\label{lem:inequality_alpha}
For any $ |x| \le 1/2$ and $ 0 < \alpha < 1 $, it holds that
\begin{align*}
(1 + x)^\alpha \ge 1 + \alpha x + \frac{\alpha(\alpha - 1)}{2} x^2 - \frac{4\alpha(\alpha - 1)(\alpha - 2)}{3}  |x|^3.
\end{align*}
\end{lemma}

\begin{proof}
Using Taylor's expansion with the integral form of the remainder, we have $(1 + x)^\alpha = 1 + \alpha x + \frac{\alpha(\alpha - 1)}{2} x^2 + R_2(x)$, where the remainder term
$$
R_2(x) = \int_0^x \frac{\alpha(\alpha - 1)(\alpha - 2)}{2} (1 + t)^{\alpha - 3} (x - t)^2 \, dt.
$$
To bound $ |R_2(x)| $, observe that for $ |x| \le 1/2 $ and $ 0 < \alpha < 1 $, the integrand satisfies $(1 + t)^{\alpha - 3} \le (1 - |x|)^{\alpha - 3} \le 2^{3 - \alpha}.$ Therefore, we have
$$
|R_2(x)| \le \int_0^{|x|} \frac{\alpha(\alpha - 1)(\alpha - 2)}{2} 2^{3 - \alpha} (|x| - t)^2 \, dt \le \frac{\alpha(\alpha - 1)(\alpha - 2)}{6} 2^{3 - \alpha} |x|^3.
$$
Combining this with the Taylor expansion, we obtain the desired inequality.
\end{proof}

\begin{lemma}\label{lem:2eqs}
Let $x$ and $y$ satisfy the following equation system:
\begin{align*}
\begin{cases}
    &x^{2L-1} - \sqrt{\lambda} yx^{L-1} + \lambda x= 0, \notag \\
    &(2L -1)x^{2L-2} - \sqrt{\lambda}(L-1)yx^{L-2} + \lambda = 0. \notag
\end{cases}
\end{align*}
Then we have
$
    y = \left( ( \frac{L-2}{L} )^{{L}/{2(L-1)}} + ( \frac{L}{L-2} )^{{(L-2)}/{2(L-1)}} \right)  \lambda^{{1}/{2(L-1)}}.
$
\end{lemma}
\begin{proof}
Eliminating the term with \( y \) gives $L x^{2L-1} = \lambda (L-2) x$, implying
\( x = \big( \tfrac{\lambda(L-2)}{L} \big)^{1/(2L-2)}. \)
Substituting this into the first equation gives the result.
\end{proof}
\begin{lemma}
\label{coro:completelocal}
Let $y_i$ for each $i \in [d_Y]$ be defined in \eqref{eq:Y}. For $L\ge 3$ and a function
$ g(x; y_i) := (x^L - \sqrt{\lambda} y_i)^2 + \lambda Lx^2$ for all $i \in [d_Y],$ the following statements hold: \\
(i)  For each $i \in [d_Y]$, $0$ is a local minimizer of $g(x;y_i)$.\\
(ii) For each $i \in [d_Y]$ and $\overline{x}(y_i)\in \mathcal{S}_1$, we have $\overline{x}(y_i)$ is a local minimizer
of $g(x;y_i)$.
\end{lemma}
\begin{proof}
(i) One can verify that $\partial_x g(0,y_i) = 0$ and $\partial_{xx} g(0,y_i) > 0$, implying $0$ is a local minimum of $g(x; y_i)$.
(ii) Since $\partial_x g(\overline{x}(y_i); y_i) = 2L f(\overline{x}(y_i); y_i) = 0$, it follows from { \Cref{lem:derprop}(iii)} that $\overline{x}(y_i)$ is also a local minimum of $g(x; y_i)$ for all $i \in [d_Y]$.
\end{proof}
\end{document}